\documentclass[11pt,letterpaper]{amsart}
\usepackage[utf8]{inputenc}
\usepackage{amsmath,amssymb,amsfonts,amsthm}
\usepackage{lmodern}
\usepackage[hidelinks]{hyperref}

\def\le {\leqslant}
\def\ge {\geqslant}
\DeclareMathOperator{\supp}{supp}

\newtheorem{theorem}{Theorem}[section]
\newtheorem{corollary}[theorem]{Corollary}
\newtheorem{lemma}[theorem]{Lemma}
\newtheorem{conjecture}[theorem]{Conjecture}
\newtheorem{proposition}[theorem]{Proposition}

\theoremstyle{definition}
\newtheorem{definition}[theorem]{Definition}
\newtheorem{remark}[theorem]{Remark}

\numberwithin{equation}{section}

\title[Oscillatory integrals with polynomial phase]{Oscillatory integrals with polynomial phase and regularity of distributions}

\author{Egor Kosov}
\address{E. Kosov, Centre de Recerca Matem\`atica, Campus de Bellaterra, Edifici~C 08193
	Bellaterra (Barcelona), Spain.}
\email{kosoved09@gmail.com}

\begin{document}

\subjclass[2020]{Primary 52A05, 42B10, Secondary 26B30, 26D05}
\keywords{Oscillatory integral, van der Corput lemma, s-concave measure, log-concave measure, polynomial}

\begin{abstract}
We obtain dimension-free estimates for the modulus of continuity of densities of polynomial images of $s$-concave and product measures. As a consequence, we settle a conjecture of A.~Carbery and J.~Wright (2001) on sharp upper bounds for oscillatory integrals over convex sets with polynomial phase.
\end{abstract}

\maketitle

\section{Introduction}

\subsection{Oscillatory integrals and van der Corput lemma}
Estimates of oscillatory integrals of the form
\begin{equation}\label{integral}
\int_{\mathbb{R}^n}e^{itf(x)}g(x)\, dx
\end{equation}
are a classical topic in modern analysis and have been studied extensively (see, e.g., the monographs \cite{Gr14} and \cite{Stein}).
The central objective is to understand the asymptotic decay of the oscillatory integral \eqref{integral} as $t \to \infty$.
One of the first classical results in this direction is the so-called van der Corput lemma (see \cite{AKCh}, \cite[Section~2.6.2]{Gr14}, or \cite[Chapter VIII]{Stein}), which asserts that
$$
\Bigl|\int_a^be^{itf(x)}\, dx\Bigr|
\le C_k |t|^{-1/k}\quad \forall t\in \mathbb{R}\setminus\{0\}
$$
provided that $f^{(k)}\ge 1$ for some $k\ge2$
(in the case $k=1$, one also needs to assume the monotonicity of $f'$).
A.~Carbery, M.~Christ, and J.~Wright in \cite{CCW99} studied multidimensional analogues of the van der Corput lemma and obtained the following power decay estimate:

$$
\Bigl|\int_{[0, 1]^n}e^{itf(x)}\, dx\Bigr|
\le C(\alpha, n)|t|^{-\varepsilon(\alpha, n)}\quad \forall t\in \mathbb{R}\setminus\{0\}
$$
provided that $\frac{\partial^\alpha f}{\partial x^\alpha}\ge 1$, where $\alpha=(\alpha_1, \ldots, \alpha_n)$ and
$\alpha_j\ge 2$ for some index $j\in \{1,\ldots, n\}$.
In the case of a polynomial phase $f$, a stronger result was established in \cite{CCW99}.
Namely, let $\mathcal{P}_d(\mathbb{R}^n)$ denote the space of all algebraic polynomials in $n$ variables of degree at most $d$.
Then, for every $f \in \mathcal{P}_d(\mathbb{R}^n)$, the following two estimates hold for all $t \in \mathbb{R}\setminus\{0\}$ (see Theorem~7.2 and Corollary~7.3 in \cite{CCW99}):
\begin{equation}\label{Th-CCW-1}
\Bigl|\int_{[0, 1]^n}e^{itf(x)}\, dx\Bigr|
\le C(d, n)|t|^{-1/|\alpha|}
\end{equation}
provided that $\frac{\partial^\alpha f}{\partial x^\alpha}(x)\ge 1$,
where $|\alpha|=\alpha_1+\ldots+\alpha_n$,
and 
\begin{equation}\label{Th-CCW-2}
\Bigl|\int_{[0, 1]^n}e^{itf(x)}\, dx\Bigr|
\le C(d, n)\Bigl(\sum_{0<j_1+\ldots+j_n\le d}|a_{j_1,\ldots, j_n}|\Bigr)^{-1/d}|t|^{-1/d}.
\end{equation}

\noindent
Subsequently, this and related estimates in the multidimensional setting were investigated by A.~Carbery and J.~Wright~\cite{CW02}, M.~Christ, X.~Li, T.~Tao, C.~Thiele~\cite{CLTT05}, 
M.~Gilula, P.T.~Gressman, L.~Xiao~\cite{GGX18}, 
%M.~Gilula, K.~O'Neill, L.~Xiao~\cite{GOX21}, 
P.T.~Gressman, L.~Xiao~\cite{GX16}, 
D.H.~Phong, E.M.~Stein, J.~Sturm \cite{PSS01}, and many others.

It is well known that estimates of the oscillatory integral \eqref{integral} are closely connected with the measure of the sublevel sets of the function $f$.
For polynomials on convex domains, sharp estimates of this type were established by A.~Carbery and J.~Wright~\cite{CarWr} (see also \cite{NSV}). Specifically,
there exists an absolute constant $C>0$ such that for every $n, d\in \mathbb{N}$, 
every convex body $K\subset \mathbb{R}^n$ of volume~$1$, 
and every $f\in\mathcal{P}_d(\mathbb{R}^n)$, one has
\begin{equation}\label{CW-est}
\|f\|_{L^2(K)}^{1/d}\lambda_n(x\in K\colon |f(x)|\le t)
\le C\min(d, n) t^{1/d}\quad \forall t>0,
\end{equation}	
where $\lambda_n$ denotes the standard Lebesgue measure on $\mathbb{R}^n$.
Motivated by this estimate, they formulated the following conjecture:

\begin{conjecture}[see {\cite[Section 6]{CarWr}}]\label{T-osc}
There exists an absolute constant $C\!>\!0$
such that
for every polynomial $f\in \mathcal{P}_d(\mathbb{R}^n)$ satisfying the normalization conditions
$$
\int_{[0, 1]^n}f(x)\, dx=0,\quad \int_{[0, 1]^n}|f(x)|\, dx=1,
$$
the inequality 
\begin{equation}\label{eq-CW-main}
\Bigl|\int_{[0, 1]^n} e^{itf(x)}\, dx\Bigr|
\le \frac{C\min\{d, n\}}{|t|^{1/d}}\quad \forall t\in\mathbb{R}\setminus \{0\}
\end{equation}
holds.
\end{conjecture}

In their paper, Carbery and Wright proved only an averaged version of the estimate \eqref{eq-CW-main}.
A partial result, with an additional dimensional factor $n^{1/2d}$, was later obtained by I.~Parissis in his PhD thesis~\cite[Theorem~2.12]{Par07}.
Furthermore, 
an estimate of the form~\eqref{eq-CW-main} with a dimension-free constant follows 
from~\cite{Kos}.
%Later, the estimate \eqref{eq-CW-main} with a dimension-free constant follows from~\cite{Kos}.
%The first dimension-free result in this direction follows from~\cite{Kos}.
%An estimate with a dimension-free constant follows from~\cite{Kos}.
More recently, I.~Glazer and D.~Mikulincer~\cite{GM22} established the required estimate, 
with the desired dependence on $d$, under substantially more restrictive 
normalization assumptions on the coefficients of the polynomial 
(see below for the discussion of their result stated in Theorem~\ref{T-GM}).

\vskip .05in

In this paper, we confirm the Carbery--Wright conjecture.

\subsection{Connection to fractional regularity of image measures}

To obtain upper bounds for the oscillatory integral \eqref{integral}, it is sufficient to estimate
$$
\int_{\mathbb{R}^n}\cos(tf(x))g(x)\, dx
\hbox{ and }
\int_{\mathbb{R}^n}\sin(tf(x))g(x)\, dx.
$$
Both of these integrals are of the form
$$
\int_{\mathbb{R}^n}\varphi'(f(x))g(x)\, dx,
$$
where $\varphi\in C_b^\infty(\mathbb{R})$,
$\|\varphi\|_\infty\le |t|^{-1}$, and $\|\varphi'\|_\infty\le 1$. 
This implies that, for the purpose of studying the behavior of the oscillatory integral \eqref{integral}, it is sufficient to estimate the supremum
\begin{equation}\label{eq-sup}
\sup\Bigl\{
\int_{\mathbb{R}^n}\varphi'(f(x))g(x)\, dx\colon \varphi\in C_b^\infty(\mathbb{R}), \|\varphi\|_\infty\le \varepsilon, \|\varphi'\|_\infty\le1 \Bigr\}.
\end{equation}
If $\mu$ denotes the measure with density $g$ with respect to the Lebesgue measure, 
then
$$
\int_{\mathbb{R}^n}\varphi'(f(x))g(x)\, dx
=
\int_{\mathbb{R}^n}\varphi'(f(x))\, \mu(dx)
=\int_{\mathbb{R}}\varphi'(s)\, \mu\circ f^{-1}(ds),
$$
where $\mu\circ f^{-1}$ stands for the image measure of $\mu$ under the mapping $f$,
that is,
$$
\mu\circ f^{-1}(A) = \mu(f^{-1}(A))
$$
for all Borel sets $A$. 
Motivated by this observation, we introduce the following two functionals 
(see~\cite{Kos-FCAA}).
For $\varrho\in L^1(\mathbb{R})$ and $\varepsilon>0$, define
$$
\sigma(\varrho, \varepsilon):=\sup\Bigl\{\int_{\mathbb{R}}\varphi'(t)\varrho(t)\, dt\colon \varphi\in C_b^\infty(\mathbb{R}), \|\varphi\|_\infty\le \varepsilon, \|\varphi'\|_\infty\le1 \Bigr\}.
$$
Similarly, for a bounded Borel measure $\nu$ on $\mathbb{R}$
and $\varepsilon>0$, define
$$
\sigma(\nu, \varepsilon):=\sup\Bigl\{\int_{\mathbb{R}}\varphi'(t)\, \nu(dt)\colon \varphi\in C_b^\infty(\mathbb{R}), \|\varphi\|_\infty\le \varepsilon, \|\varphi'\|_\infty\le1 \Bigr\}.
$$

\noindent
Therefore, the supremum \eqref{eq-sup} coincides with $\sigma(\mu\circ f^{-1}, \varepsilon)=\sigma(\varrho_f, \varepsilon)$,
where $\varrho_f$ denotes the density of the image measure $\mu\circ f^{-1}$, whenever it exists.
It is known that the functional 
$\sigma(\varrho, \cdot)$ describes the regularity properties of the function 
$\varrho$. Namely, the following two-sided estimate holds (see \cite[Theorem~2.1]{Kos-MS}):
\begin{equation}\label{T-equiv}
2^{-1}\omega(\varrho, \varepsilon)\le \sigma(\varrho, \varepsilon)\le 6\,\omega(\varrho, \varepsilon),\quad 
\forall\varrho\in L^1(\mathbb{R}),\quad \forall \varepsilon>0,
\end{equation}
where
$$
\omega(\varrho, \varepsilon): = \sup\limits_{|h|\le \varepsilon}\int_{\mathbb{R}}
|\varrho(x+h) - \varrho(x)|\, dx
$$ 
is the classical integral modulus of continuity 
(see \cite[Chapter 2, \S 7]{DL93} or~\cite{KT20}).

We recall that the modulus of continuity 
$\omega(\varrho, \cdot)$ is commonly used to define certain function spaces.
For example, see \cite{BIN} and \cite{Stein}, the Nikolskii–Besov space 
$B^\alpha_{1,\infty}(\mathbb{R})$, $\alpha\in (0,1)$, consists of all functions 
$\varrho\in L^1(\mathbb{R})$ such that 
$$
\|\varrho\|_{\rm \dot{B}^\alpha_{1,\infty}(\mathbb{R})}:=\sup_{\varepsilon>0}\varepsilon^{-\alpha}\omega(\varrho, \varepsilon) <\infty.
$$
When $\alpha=1$, we recover the definition of the class of functions of bounded variation~$BV(\mathbb{R})$ with the semi-norm
$$
\|\varrho\|_{\rm \dot{BV}(\mathbb{R})}:=
\sup_{\varepsilon>0}\varepsilon^{-1}\omega(\varrho, \varepsilon).
$$

To summarize, continuing the research initiated in \cite{BKZ}, \cite{Kos}, and \cite{Kos-FCAA},
we address a more general problem and study {\it dimension-free} upper bounds for the modulus 
$\sigma(\mu\circ f^{-1}, \cdot)$ of image measures
$\mu\circ f^{-1}$, where 
$f\colon \mathbb{R}^n\to\mathbb{R}$ is a polynomial mapping and $\mu$
is a suitably regular measure on $\mathbb{R}^n$.
In particular, we prove the following result.

\begin{theorem}\label{T-main-reg}
There exists an absolute constant 
$C>0$ such that for all 
$n, d\in \mathbb{N}$, for every convex body 
$K\subset \mathbb{R}^n$ of volume $1$, 
and for every  polynomial
$f\in \mathcal{P}_d(\mathbb{R}^n)$, one has
$$
\sigma(\mu_K\circ f^{-1}, \varepsilon)
\le 
\frac{C\min\{d, n\}} {(\mathbb{D}_K f)^{1/2d}}\cdot \varepsilon^{1/d}\quad \forall \varepsilon>0,
$$
where $\mu_K$ denotes the restriction of the Lebesgue measure to $K$
and
$$
\mathbb{D}_Kf:=\int_K\Bigl(f(x) - \int_Kf(y)\, dy\Bigr)^2\, dx.
$$
In other words, when $f$ is a
non-constant polynomial of degree at most $d$,
the distribution density $\varrho_f$ of the measure
$\mu_K\circ f^{-1}$ belongs to the 
Nikolskii–Besov space 
$B^{1/d}_{1,\infty}(\mathbb{R})$,
and
$$
\|\varrho_f\|_{\rm \dot{B}^{1/d}_{1,\infty}(\mathbb{R})}
\le 
\frac{C\min\{d, n\}} {(\mathbb{D}_K f)^{1/2d}}.
$$
\end{theorem}

\noindent
Conjecture~\ref{T-osc} is an immediate consequence of this result. 

Since the Carbery--Wright estimate~\eqref{CW-est} is sharp up to a constant
factor (see the discussion after Theorem 2 in \cite{CarWr}), 
and since estimates of the modulus 
$\sigma(\mu\circ f^{-1}, \cdot)$
imply estimates for the measure of the sublevel sets 
(see Theorem~\ref{T-meas}), 
Theorem~\ref{T-main-reg} is also sharp up to a constant factor.

To prove Theorem~\ref{T-main-reg}, in Section~\ref{sec-one-dim} we first examine the one-dimensional case and establish a complete analog of the van der Corput lemma for the regularity of images of measures with densities of bounded variation. Namely,
we prove the following theorem.

\begin{theorem}\label{T-1d}
There exists an absolute constant $C>0$ such that for every $k \in \mathbb{N}$, $k \ge 2$, for every probability measure $\nu$ on $\mathbb{R}$ with a density $\varrho$ of bounded variation, and for every $f \in C^\infty(\mathbb{R})$ satisfying $f^{(k)}(t) \ge 1$ for all $t \in \mathbb{R}$, one has 
$$
\sigma\bigl(\nu\circ f^{-1}, \varepsilon\bigr)
\le 
Ck \|\varrho \|_{\rm \dot{BV}(\mathbb{R})}\cdot \varepsilon^{1/k}\quad \forall \varepsilon>0.
$$
\end{theorem}

In Section~\ref{sec-mult-dim}, we move on to the multidimensional setting, following the ideas of \cite{Kos} and \cite{Kos-IMRN}. First, in Corollary~\ref{cor-pol}, by applying the one-dimensional estimate along each fixed direction, we obtain a general bound for the modulus of continuity $\sigma(\mu \circ f^{-1}, \cdot)$, where $f \in \mathcal{P}_d(\mathbb{R}^n)$ and $\mu$ is a measure with a density $\varrho$ of bounded variation. This estimate involves the total variation norm of the directional derivative of $\varrho$ and the measure of the sublevel sets of the directional derivative of~$f$. To control these two parameters, we restrict ourselves to the class of so-called $s$-concave measures with $s \ge 0$ (see Definition~\ref{def-s-conc} below). In particular, the classical Brunn–Minkowski inequality implies that uniform distributions on convex sets are $1/n$-concave. 
We then use several known properties of 
$s$-concave measures (in particular, of log-concave measures) to control the norm of the directional derivative. Furthermore, we employ a suitable reformulation of the Carbery–Wright inequality~\eqref{CW-est}, together with the dimensional Poincaré inequality, to estimate the measure of the sublevel sets.
All of this implies a dimensional version of Theorem~\ref{T-main-reg} that is valid for the class of $s$-concave measures (see Lemma~\ref{lem-dim}). Finally, to complete the proof and control the dimensional dependence, we apply the so-called localization lemma from~\cite{FrGue} to reduce the estimate in the general multidimensional case to estimates in low-dimensional settings.
This leads to Theorem~\ref{CW-T}, which provides a dimension-free estimate for $\sigma(\mu \circ f^{-1}, \cdot)$, valid for all $s$-concave measures~$\mu$, thereby implying Theorem~\ref{T-main-reg}.

\subsection{The case of product measures}

Apart from uniform distributions on convex subsets of 
$\mathbb{R}^n$ and their generalizations (see Definition \ref{def-s-conc}), we also consider product measures and study the regularity properties of their polynomial images. In order to formulate the results, we need to introduce the following auxiliary definitions.

\begin{definition}\label{def-pol}
For $d,m\in\mathbb{N}$, let $\mathcal{P}_{d,m}(\mathbb{R}^n)\subset \mathcal{P}_d(\mathbb{R}^n)$ 
denote the space of all algebraic polynomials of total degree at most~$d$ whose individual degree does not exceed~$m$, that is, functions $f$ of the form
$$
f(x):=\sum_{\substack{j_1+\ldots+j_n\le d\\ \max\{j_1, \ldots, j_n\}\le m}}a_{j_1,\ldots, j_n}x_1^{j_1}\ldots x_n^{j_n}.
$$	
\end{definition}

\noindent
In particular,
$\mathcal{P}_d(\mathbb{R}^n) = \mathcal{P}_{d,d}(\mathbb{R}^n)$.

\begin{definition}
For a polynomial $f\in\mathcal{P}_d(\mathbb{R}^n)$
 of the form
$$
f(x):=\sum_{j_1+\ldots+j_n\le d}a_{j_1,\ldots, j_n}x_1^{j_1}\ldots x_n^{j_n},
$$
let
$$
d(f):=\max\{j_1+\ldots+j_n\colon 
a_{j_1,\ldots, j_n}\ne 0\},
$$
$$
[f]_2:=\Bigl(\sum_{j_1+\ldots+j_n= d(f)}a_{j_1,\ldots, j_n}^2\Bigr)^{1/2},
$$
$$
[f]_\infty:= \max\{|a_{j_1,\ldots, j_n}|\colon j_1+\ldots+j_n= d(f)\}.
$$
\end{definition}

We can now state two main results concerning the regularity of polynomial images of product measures.

\begin{theorem}\label{reg-prod-1}
There exists an absolute constant $C>0$ such that for all $d,n\in\mathbb{N}$, for any probability measures $\nu_1,\ldots,\nu_n$ on $\mathbb{R}$, each with a density $\varrho_j$ of bounded variation, and for any non-constant polynomial $f\in\mathcal{P}_d(\mathbb{R}^n)$, one has	
$$
\sigma\bigl((\otimes_{j=1}^n \nu_j)\circ f^{-1}, \varepsilon\bigr)
\le 
C\min\{d, n\}\bigl(1+\max_{1\le j\le n}\|\varrho_j\|_{\rm \dot{BV}(\mathbb{R})}\bigr)[f]_2^{-1/d}
\varepsilon^{1/d}
$$for every $\varepsilon>0$.
\end{theorem}

\begin{theorem}\label{reg-prod-2}
For $d,m \in \mathbb{N}$ with $d \ge m$, there exists a constant $C(d,m)$, depending only on $d$ and $m$, such that for any probability measures $\nu_1, \ldots, \nu_n$ on~$\mathbb{R}$, each with a density $\varrho_j$ of bounded variation, and for any non-constant polynomial~$f\in\mathcal{P}_{d, m}(\mathbb{R}^n)$,
one has
\begin{align*}
\sigma\bigl((&\otimes_{j=1}^n \nu_j)\circ f^{-1}, \varepsilon\bigr)
\\
&\le C(d, m)\bigl(1+\max_{1\le j\le n}\|\varrho_j\|_{\rm \dot{BV}(\mathbb{R})}\bigr)^{d/m} [f]_\infty^{-1/m} \varepsilon^{1/m}\bigl(|\ln ([f]_\infty^{-1}\varepsilon)|^{d-m}+1\bigr)
\end{align*}
for every $\varepsilon>0$.
\end{theorem}

The proofs of Theorems~\ref{reg-prod-1}~and~\ref{reg-prod-2}
proceed in two steps. 
First, in Section~\ref{sec-cube}, we establish these theorems for the unit cube $Q^n := [-\frac{1}{2}, \frac{1}{2}]^n$, that is, in the case when each $\nu_j$ is the uniform distribution on $[-\frac{1}{2}, \frac{1}{2}]$. To this end, we use Theorem~\ref{T-main-reg} and apply 
\cite[Theorem 1]{GM22} to control the variance $\mathbb{D}_{Q^n} f$. Moreover, to prove Theorem~\ref{reg-prod-2} in this case, we follow the approach of~\cite{Kos-FA} and employ an inductive argument based on the relation between the measure of sublevel sets and the regularity of the distribution (see Theorem~\ref{T-meas}). Second, in Section~\ref{sec-prod}, we complete the proofs of Theorems~\ref{reg-prod-1} and~\ref{reg-prod-2} in a slightly more general form (see Theorems~\ref{Coeff-T-1} and~\ref{Coeff-T-2}) by reducing the general case to that of the unit cube using the technique developed in~\cite{BChG}.

\vskip .1in

Both of the above theorems imply a corresponding estimate for the oscillatory integrals.
Namely, let $\nu_1, \ldots, \nu_n$ 
be probability measures
on $\mathbb{R}$, each with a density~$\varrho_j$ of bounded variation.
Then Theorem~\ref{reg-prod-1} yields
\begin{equation}\label{T-osc-prod-1}
\Bigl|\int_{\mathbb{R}^n}e^{it f(x)}\, \nu_1(dx_1)\ldots\nu_n(dx_n)\Bigr|
\le (1+\max_{1\le j\le n}\|\varrho_j\|_{\rm \dot{BV}(\mathbb{R})}) \frac{C\min\{d, n\}}{([f]_2 |t|)^{1/d}}
\end{equation}
for every $t\in\mathbb{R}\setminus\{0\}$
and every
non-constant $f\in\mathcal{P}_d(\mathbb{R}^n)$.
Similarly, 
Theorem~\ref{reg-prod-1} implies
\begin{align}\label{T-osc-prod-2}
\Bigl|\int_{\mathbb{R}^n}&e^{it f(x)}\, \nu_1(dx_1)\ldots\nu_n(dx_n)\Bigr|
\\
&\le
\bigl(1+\max_{1\le j\le n}\|\varrho_j\|_{\rm \dot{BV}(\mathbb{R})}\bigr)^{d/m} \frac{C(d, m)}{ ([f]_\infty|t|)^{1/m}
}\bigl(|\ln ([f]_\infty |t|)|^{d-m}+1\bigr)\nonumber
\end{align}
for every $t\in\mathbb{R}\setminus\{0\}$
and every
non-constant $f\in\mathcal{P}_{d,m}(\mathbb{R}^n)$.
In particular, these estimates provide a dimension-free counterpart of the inequality~\eqref{Th-CCW-2} and generalize the following recent result by I.~Glazer and D.~Mikulincer.

\begin{theorem}[see {\cite[Theorem 5]{GM22}}]\label{T-GM}
There exists an absolute constant $C>0$ such that for any log-concave probability measure $\nu$ on $\mathbb{R}$ (that is, a measure with density $e^{-V}$, where $V\colon \mathbb{R}\to(-\infty,+\infty]$ is convex) satisfying	
$$
\int_{\mathbb{R}} x\, \nu(dx) = 0, \quad 
\int_{\mathbb{R}} x^2\, \nu(dx) = 1,
$$
and for any polynomial $f\in \mathcal{P}_d(\mathbb{R}^n)$, one has
$$
\Bigl|\int_{\mathbb{R}^n}e^{it f(x)}\, \nu(dx_1)\ldots\nu(dx_n)\Bigr|
\le \frac{Cd}{([f]_\infty |t|)^{1/d}}\quad \forall t\in\mathbb{R}\setminus\{0\}.
$$
\end{theorem}

It is worth noting  that the estimate~\eqref{T-osc-prod-1}
not only replaces the quantity 
$[f]_\infty$ with the larger parameter 
$[f]_2$ compared to the theorem above,
but also applies to a broader class of measures (see~\eqref{Krug} and, e.g.,~\eqref{isotr-est}). Moreover, using the smaller parameter $[f]_\infty$ in the estimate~\eqref{T-osc-prod-2} yields a sharper decay rate that depends on the individual degree of the polynomial.
We also point out that Theorem~\ref{reg-prod-2} and the estimate~\eqref{T-osc-prod-2} generalize the results obtained for Gaussian measures in~\cite{GPU} and~\cite{Kos-FA}.

\subsection{Convergence of random variables}

The dimension-free nature of the estimates in Theorems~\ref{T-main-reg}, \ref{reg-prod-1}, and~\ref{reg-prod-2} is particularly useful for studying convergence in total variation distance of sequences of random variables that converge in distribution.
This problem has been extensively studied in recent years both for sequences of multiple stochastic integrals, that is, for random variables of the form $f(X)$, where $X$ is an infinite-dimensional Gaussian random element (for example, a Wiener process) and $f$ is a measurable polynomial mapping (see~\cite{BKZ}, \cite{Kos-FCAA}, \cite{NNP}, \cite{NP13}), as well as in more general settings (see~\cite{BC14}, \cite{BC17}, \cite{BCP}, \cite{Kos}, \cite{Kos-Adv}).
The connection between the regularity properties of distributions and upper bounds for the total variation distance is given by the following inequality 
(see~\cite[Theorem~3.2]{BKZ} or~\cite[Remark~3.1]{Kos-FCAA}):
\begin{equation}\label{eq-TV-K}
{\rm d_{TV}}(\xi_1, \xi_2)
\le C \bigl(\max\bigl\{\|\varrho_1\|_{\rm \dot{B}^\alpha_{1, \infty}}, \|\varrho_2\|_{\rm \dot{B}^\alpha_{1, \infty}}\bigr\}\bigr)^{\frac{1}{1+\alpha}}
{\rm d_{K}}(\xi_1, \xi_2)^{\frac{\alpha}{1+\alpha}}
\end{equation}
for two random variables $\xi_1$ and $\xi_2$ on $\mathbb{R}$ with densities $\varrho_1$ and $\varrho_2$, respectively.
Here ${\rm d_{TV}}(\xi_1, \xi_2)$ 
denotes the total variation distance between
the distributions of $\xi_1$ and $\xi_2$, that is,
$$
{\rm d_{TV}}(\xi_1, \xi_2)  := \sup\Bigl\{\mathbb{E}\bigl(\varphi(\xi_1) - \varphi(\xi_2)\bigr)\colon \ \varphi\in C_b^\infty(\mathbb{R}),
\ \|\varphi\|_\infty \le 1 \Bigr\}
$$
and 
${\rm d_{K}}(\xi_1, \xi_2)$ denotes the Kantorovich distance,
$$
{\rm d_{K}}(\xi_1, \xi_2):=
\sup\Bigl\{\mathbb{E}\bigl(\varphi(\xi_1) - \varphi(\xi_2)\bigr)\colon
\varphi\in C_b^\infty(\mathbb{R}),\ \|\varphi'\|_\infty\le1\Bigr\},
$$
defined on the space of all 
random variables with finite first moment, on which the Kantorovich distance metrizes convergence in distribution.

The estimate~\eqref{eq-TV-K}, together with Theorem~\ref{T-main-reg}, clarifies the dependence on the degree of the polynomials in the main result from~\cite{Kos}, yielding the bound
$$
{\rm d_{ TV}}(f(X), g(X))
\le CdA^{-\frac{1}{d+1}}{\rm d_{K}}(f(X), g(X))^{\frac{1}{d+1}}
$$
for a log-concave random vector $X$
(finite- or infinite-dimensional)
and for polynomials $f$ and $g$
of degree at most $d$ satisfying
$$
\mathbb{D}f(X), \mathbb{D}g(X)\ge A^2>0,
$$
where $\mathbb{D}$ denotes the variance of a random variable. In particular, the estimate remains valid for multiple stochastic integrals, where $X$ is a Gaussian random element. 

Similarly, combining the estimate~\eqref{eq-TV-K} with Theorem~\ref{reg-prod-1}, we obtain a completely new result in the case of a random vector $X = (X_1, \ldots, X_n)$ with independent coordinates $X_j$, each having a density $\varrho_j$ of bounded variation satisfying $\|\varrho_j\|_{\rm \dot{BV}(\mathbb{R})} \le 1$. Namely, the following bound for the total variation distance holds:
$$
{\rm d_{ TV}}(f(X), g(X))
\le C\min\{d, n\}A^{-\frac{1}{d+1}}{\rm d_{K}}(f(X), g(X))^{\frac{1}{d+1}}
$$
for all $f, g \in \mathcal{P}_d(\mathbb{R}^n)$ satisfying
$$
[f]_2, [g]_2\ge A.
$$ 
We expect that this estimate, together with Theorem~\ref{reg-prod-1}, will be useful for obtaining new upper bounds on the rate of convergence in the total variation distance in the so-called invariance principle for polynomials (see~\cite{BC19}, \cite{KosZh}, \cite{MOO}, \cite{NP15}).

\section{Preliminaries and notation}\label{sec-prelim}

In this section, we introduce the definitions and notation used throughout the paper and review some known results that will be needed later.

\subsection{Notation for constants}

Throughout the paper, $C, c > 0$ denote absolute constants that may vary from line to line.
When needed, we distinguish different occurrences by subscripts, writing $C_1, C_2$, etc. These subscripts serve merely as labels and do not indicate parameter dependence.
Dependence on parameters will be denoted by $C(n)$ or $C(m,d)$, where $n$ denotes the dimension of the space, $d$ the total degree of a polynomial, and $m$ its individual degree.

\subsection{Measures and functions of bounded variation}

Let $C_0^\infty(\mathbb{R}^n)$ denote the space of all smooth functions with compact support, and $C_b^\infty(\mathbb{R}^n)$ denote the space of all bounded smooth functions with bounded derivatives of every order. The standard Euclidean inner product of two vectors $x, y \in \mathbb{R}^n$ is denoted by $\langle x, y \rangle$, and the corresponding Euclidean norm of $x \in \mathbb{R}^n$ is denoted by $|x|$. We denote the standard Lebesgue measure on $\mathbb{R}^n$ by $\lambda_n$, and when $n=1$, we write $\lambda$ instead of~$\lambda_1$.

Let $\mu$ be a bounded (possibly signed) Borel measure on $\mathbb{R}^n$. Its total variation norm is defined by the equality
$$
\|\mu\|_{\rm TV}  := \sup\Bigl\{\int_{\mathbb{R}^n} \varphi \, d\mu, \ \varphi\in C_b^\infty(\mathbb{R}^n),
\ \|\varphi\|_\infty \le1 \Bigr\},
$$
where
$$
\|\varphi\|_\infty:= \sup_{x\in \mathbb{R}^n}|\varphi(x)|.
$$
The image of a measure $\mu$ under a measurable mapping $f\colon \mathbb{R}^n \to \mathbb{R}$ is the measure $\mu \circ f^{-1}$, defined by
$$
\mu\circ f^{-1} (A) = \mu\bigl(f^{-1}(A) \bigr) \quad \text{for every Borel set } A\subset\mathbb{R}.
$$

A function $\varrho \in L^1(\mathbb{R}^n)$ is said to be of bounded variation (see \cite[Definition~3.1]{AFP}) if, for every $\theta \in \mathbb{R}^n$, there exists a bounded Borel measure $D_\theta \varrho$ such that
$$
\int_{\mathbb{R}^n}\varrho(x)\partial_\theta\varphi(x)\, dx=
-\int_{\mathbb{R}^n}\varphi(x)\, D_\theta \varrho(dx)\quad 
\forall \varphi\in C_0^\infty(\mathbb{R}^n).
$$
Let $BV(\mathbb{R}^n)$ denote the collection of all such functions.
We note that any Sobolev function $\varrho \in W^{1,1}(\mathbb{R}^n)$ belongs to $BV(\mathbb{R}^n)$.
It is also clear that a function $\varrho \in L^1(\mathbb{R})$ belongs to $BV(\mathbb{R})$ if and only if 
$\sup\limits_{\varepsilon>0} \varepsilon^{-1}\sigma(\varrho, \varepsilon) < \infty,$
and equivalently, if 
$\sup\limits_{\varepsilon>0} \varepsilon^{-1}\omega(\varrho, \varepsilon) < \infty.$
Moreover, by the estimate~\eqref{T-equiv},
$$
2^{-1}\|\varrho\|_{\rm \dot{BV}(\mathbb{R})}
\le 
\sup\limits_{\varepsilon>0} \varepsilon^{-1}\sigma(\varrho, \varepsilon)
=\|D_1\varrho\|_{\rm TV}\le 6 
\|\varrho\|_{\rm \dot{BV}(\mathbb{R})}.
$$

We will need the following result, which relates the measure of a given set to the regularity properties of the measure.

\begin{theorem}[see {\cite[Lemma 2.4]{Kos}} or {\cite[Corollary 2.2]{Kos-MS}}]\label{T-meas}
For every Borel set $A$ on $\mathbb{R}$ and every probability Borel measure $\nu$ on $\mathbb{R}$, one has
$$
\nu(A)\le \sigma(\nu, \lambda(A)).
$$
\end{theorem}

\subsection{Log-concave and $s$-concave measures}

We will use some properties of the so-called $s$-concave measures, introduced by C.~Borell in \cite{Bor74}, \cite{Bor75}.

\begin{definition}[see {\cite[Section 2.1.1]{BGVV}}]\label{def-s-conc}
Let $s\in [-\infty, \infty]$.	
A probability Borel measure $\mu$ on $\mathbb{R}^n$ is called $s$-concave if
$$
\mu\bigl(\alpha A + (1-\alpha) B\bigr)\ge \bigl(\alpha \mu^s(A) + (1-\alpha)\mu^s(B)\bigr)^{1/s}
$$	
for all compact subsets $A, B\subset \mathbb{R}^n$ with $\mu(A)\mu(B)>0$
and all $\alpha\in[0, 1]$.
\end{definition}

\begin{definition}[see {\cite[Section 2.1.1]{BGVV}}]
Let $\gamma\in[-\infty, \infty]$.
A nonnegative function $\varrho\colon \mathbb{R}^n\to[0, \infty)$
is called $\gamma$-concave
if
$$
\varrho\bigl(\alpha x + (1-\alpha)y\bigr)\ge \bigl(\alpha \varrho^\gamma(x) + (1-\alpha)\varrho^\gamma(y)\bigr)^{1/\gamma}
$$
for all points $x$ and $y$ such that $\varrho(x)\varrho(y)>0$ and for all $\alpha\in[0, 1]$.
\end{definition}

\noindent
In both definitions above, the cases $s, \gamma \in \{0, \pm\infty\}$ 
are understood in the limiting sense. In particular, for $s=0$ we recover the definition 
of a \emph{log-concave measure}, for which
$$
\mu\bigl(\alpha A + (1-\alpha) B\bigr)\ge \mu(A)^\alpha \mu(B)^{1-\alpha},
$$
and for $\gamma=0$ we recover the definition of a log-concave function $\varrho$, i.e. $\log \varrho$ is concave.  
It follows directly from the definition that an $s$-concave measure
is also $s'$-concave for every $s'<s$. In particular, every $s$-concave measure with $s \ge 0$ is log-concave. 
For this reason, our proofs will rely heavily on the known properties of log-concave measures.

For a probability Borel measure $\mu$ on $\mathbb{R}^n$, let $S(\mu)$ denote the affine subspace spanned by its support ${\rm supp}(\mu)$, and let $m_{S(\mu)}$ be the Lebesgue measure on $S(\mu)$.
There is a fundamental relation between $s$-concave measures and $\gamma$-concave functions.

\begin{theorem}[see \cite{Bor75} and {\cite[Theorem 9.1.2]{A-AGM-2}}]\label{T-Bor}
Let $\mu$ be a probability Borel measure on
$\mathbb{R}^n$, and let $k=\dim S(\mu)$. Then,
for $s\in (-\infty, 1/k]$, the measure $\mu$ is $s$-concave if and only if $\mu=\varrho \cdot m_{S(\mu)}$, where $\varrho$ is a $\gamma$-concave function with $\gamma=\frac{s}{1-sk}\in[-1/k,+\infty]$. 
\end{theorem}

\subsection{Localization lemma with several constraints}
One of the key tools associated with the class of $s$-concave measures is the {\it localization technique}, originally developed in the works of Gromov and Milman~\cite{GM87}, Lov\'asz and Simonovits~\cite{LS}, and Kannan, Lov\'asz, and Simonovits~\cite{KLS}. We will use this technique in the convenient form introduced by Fradelizi and Gu\'edon~\cite{FrGue}.

\begin{theorem}[see \cite{FrGue}]\label{loc-lem}
Let $K$ be a compact convex set in $\mathbb{R}^n$, $p\in\{1,\ldots,n\}$, and $-\infty \le s \le \tfrac{1}{p+1}$. 
Suppose $u_j\colon K \to \mathbb{R}$, $1\le j\le p$, are continuous functions on $K$. 
Define $P_{u_1,\ldots,u_p}^s(K)$ to be the set of all $s$-concave measures $\mu$ supported in $K$ such that
$$
\int_K u_j \, d\mu \ge 0, \quad j=1,\ldots,p.
$$
Let $F\colon P(K)\to \mathbb{R}$ be a convex continuous functional, where $P(K)$ denotes the space of all Borel probability measures on $K$, equipped with the weak topology (i.e., the restriction of the weak-$*$ topology on the dual of $C(K)$). 
Then 
$$
\sup\limits_{\mu\in P_{u_1,\ldots,u_p}^s(K)} F(\mu)
$$
is attained at an $s$-concave measure $\mu$ such that the affine span $S(\mu)$ of $\operatorname{supp}(\mu)$ satisfies $\dim S(\mu)\le p$. 
In particular,
$$
\sup_{\mu\in P_{u_1,\ldots,u_p}^s(K)} F(\mu)
\;\le\; \sup_{\substack{\mu\in P_{u_1,\ldots,u_p}^s(K)\\ \dim S(\mu)\le p}} F(\mu).
$$
\end{theorem}

\subsection{Isotropic constant and isotropic measures}
Another important concept that we will need is the notion of {\it the isotropic constant}.

\begin{definition}[see {\cite[Definition 2.3.11]{BGVV}}]
Let $\mu$ be an absolutely continuous probability measure on $\mathbb{R}^n$ with finite second moment, 
and let $\varrho$ denote its density. 
The \emph{isotropic constant} $L(\mu)$ of $\mu$ is defined by
$$
L(\mu):= \bigl(\sup_{x\in \mathbb{R}^n}\varrho(x)\bigr)^{\frac{1}{n}}
\bigl(\det {\rm Cov}(\mu)\bigr)^\frac{1}{2n},
$$
where ${\rm Cov}(\mu)$ is the covariance matrix of $\mu$, with entries
$$
\bigl({\rm Cov}(\mu)\bigr)_{j,k}
:= \int_{\mathbb{R}^n} x_j x_k \, \mu(dx) 
- \Bigl(\int_{\mathbb{R}^n} x_j \, \mu(dx)\Bigr)
\Bigl(\int_{\mathbb{R}^n} x_k \, \mu(dx)\Bigr).
$$
A probability measure $\mu$ on $\mathbb{R}^n$ is called \emph{isotropic} 
if ${\rm Cov}(\mu) = \mathrm{Id}$ and 
$$
\int_{\mathbb{R}^n} x_j \, \mu(dx) = 0, 
\quad \forall j \in \{1,\ldots,n\}.
$$
\end{definition}

There is a simple dimensional estimate of the isotropic constant in the log-concave case (i.e., when $s = 0$), which suffices for the purposes of this paper.
Namely (see \cite[Propositions~3.3.2 and~2.5.12]{BGVV}),  
\begin{equation}\label{isotr-est}
	L(\mu)\le c\sqrt{n}
\end{equation}
for every absolutely continuous log-concave measure $\mu$ on $\mathbb{R}^n$.

We remark that Bourgain’s hyperplane conjecture states that there exists a universal constant $c>0$, independent of the dimension~$n$, such that for every $n \in \mathbb{N}$, the isotropic constant of any absolutely continuous log-concave measure $\mu$ on~$\mathbb{R}^n$ satisfies
$$
L(\mu)\le c.
$$
This long-standing problem was resolved only recently by B.~Klartag and J.~Lehec~\cite{KL24}. 

In addition to the upper bound for the isotropic constant, we will also use the following upper bound for the density of an isotropic log-concave measure.

\begin{theorem}[see {\cite[Corollary 2.4]{Klartag07}}]\label{t-Klar}
There exist universal constants $C, c>0$ such that for every $n\in\mathbb{N}$ and every
isotropic log-concave measure $\mu$ on $\mathbb{R}^n$ with density $\varrho$, one has
$$
\varrho(x)\le \varrho(0)\, e^{Cn - c|x|}, \quad \forall x\in\mathbb{R}^n.
$$
\end{theorem}

\subsection{Poincar\'e inequality for an isotropic log-concave measure}

We will also need the following Poincar\'e inequality for isotropic log-concave measures.
\begin{theorem}[see \cite{BobkIsop} or {\cite[Theorems 2.2.8 and 2.3.3]{A-AGM-2}}]\label{Poinc}
	There exists an absolute constant $C>0$ such that for every $n\in\mathbb{N}$, 
	every isotropic log-concave measure $\mu$ on $\mathbb{R}^n$, 
	and every locally Lipschitz function $f$ on $\mathbb{R}^n$, one has
	$$
	\int_{\mathbb{R}^n}\bigl(f-\mathbb{E}_\mu f\bigr)^2\, d\mu
	\le Cn\int_{\mathbb{R}^n}|\nabla f|^2\, d\mu,
	$$
	where $\displaystyle\mathbb{E}_\mu f:=\int_{\mathbb{R}^n} f\, d\mu$.
\end{theorem}

We note that the dimensional behavior of the sharp constant in this inequality is an open problem, 
known as the Kannan--Lov\'asz--Simonovits (KLS) conjecture (see \cite[Chapter~2]{A-AGM-2}). 

\subsection{Derivatives of log-concave functions}

Finally, we need some information about the total variation norms of the derivatives of log-concave densities. 
According to~\cite[Theorem 1]{Krug}, the density $\varrho$ of any absolutely continuous log-concave measure $\mu$ on $\mathbb{R}^n$ 
is a function of bounded variation, that is, $\varrho \in BV(\mathbb{R}^n)$.  
Furthermore, the total variation of its directional derivative is given by
\begin{equation}\label{Krug}
\|D_\theta \varrho\|_{\rm TV} = 
2|\theta| \int_{\langle\theta\rangle^\bot}
\sup_{t\in\mathbb{R}}\varrho(y+t\theta)\, dy,
\end{equation}
where $\langle \theta \rangle^\bot$ denotes the orthogonal complement 
of the one-dimensional subspace $\langle \theta \rangle := \{ t \theta \colon t \in \mathbb{R} \}$ spanned by the vector $\theta$.

\section{One-dimensional van der Corput type theorem for fractional regularity}\label{sec-one-dim}

We begin with the following key one-dimensional lemma.

\begin{lemma}\label{lem-key}
Let $\varrho \in C^\infty_0(\mathbb{R})$ and $f \in C^\infty(\mathbb{R})$.  
Assume that there exist a number $r \in \mathbb{N}$ and points 
$a_0 = -\infty< a_1< \ldots< a_{r-1}< a_r = +\infty$ such that
$$
\mathbb{R}\setminus \{a_1, \ldots, a_{r-1}\} 
= \bigsqcup_{j=1}^r (a_{j-1}, a_j),
$$
and for every $j \in \{1,\ldots,r\}$, either $f''(t) \ge 0$ for all $t \in (a_{j-1},a_j)$ 
or $f''(t) \le 0$ for all $t\in(a_{j-1},a_j)$.  
Then, for every $\varphi \in C_b^\infty(\mathbb{R})$, one has
$$
\int_{\mathbb{R}} \varphi'(f(t)) \varrho(t)\, dt
\le 
C r \varepsilon^{-1} \|\varphi\|_\infty \|\varrho'\|_{L^1(\mathbb{R})}
+ \|\varphi'\|_\infty \int_{\mathbb{R}} I_{\{|f'|\le 2\varepsilon\}}\, |\varrho(t)|\, dt
$$
for a universal constant $C>1$.
\end{lemma}

\begin{proof}
We fix a function $\Phi\in C^\infty(\mathbb{R})$ 
such that
$\Phi(t) = 0$ for $t\in[-1, 1]$, $\Phi(t) = 1$ for $t\in \mathbb{R}\setminus[-2, 2]$,
and $\Phi(t)\in[0, 1]$ for every $t\in\mathbb{R}$.
For $\varepsilon>0$, let 
$$
\Phi_\varepsilon(t): = \Phi(t/\varepsilon)\quad \forall t\in \mathbb{R}.
$$

Integrating by parts, we obtain
\begin{align*}
\int_{\mathbb{R}}&\varphi'(f(t))\varrho(t)\, dt
\\
&= \int_{\mathbb{R}}\Bigl(\frac{d}{dt}\varphi(f(t))\Bigr)\frac{1}{f'(t)}\Phi_\varepsilon(f'(t))\varrho(t)\, dt
+ \int_{\mathbb{R}}\varphi'(f(t))\bigl(1 - \Phi_\varepsilon(f'(t))\bigr)\varrho(t)\, dt
\\
&=
-\int_{\mathbb{R}}\varphi(f(t))
\Bigr[\frac{f''(t)}{\varepsilon f'(t)}\Phi'(\varepsilon^{-1}f')
- \frac{f''(t)}{(f'(t))^2}\Phi_\varepsilon(f'(t))
\Bigl]\varrho(t)\, dt
\\
&-\int_{\mathbb{R}}\varphi(f(t))
\frac{1}{f'(t)}\Phi_\varepsilon(f'(t))\varrho'(t)\, dt
+ \int_{\mathbb{R}}\varphi'(f(t))\bigl(1 - \Phi_\varepsilon(f'(t))\bigr)\varrho(t)\, dt.
\end{align*}
From the definitions of $\Phi$ and $\Phi_\varepsilon$, it follows that the last expression can be bounded above by 
\begin{align*}
&\|\varphi\|_\infty
\int_{\mathbb{R}}
\Bigl[\frac{|f''(t)|}{\varepsilon |f'(t)|}\|\Phi'\|_\infty I_{\{\varepsilon\le|f'|\le 2\varepsilon\}}+
\frac{|f''(t)|}{(f'(t))^2}I_{\{|f'|\ge \varepsilon\}}
\Bigr]|\varrho(t)|\, dt
\\
&+\|\varphi\|_\infty\int_{\mathbb{R}}
\frac{1}{|f'(t)|} I_{\{|f'|\ge \varepsilon\}}|\varrho'(t)|\, dt
+ \|\varphi'\|_\infty\int_{\mathbb{R}}I_{\{|f'|\le 2\varepsilon\}}|\varrho(t)|\, dt
\\
&
\le
\|\varphi\|_\infty\bigl(2\|\Phi'\|_\infty + 1\bigr)
\|\varrho\|_\infty
\int_{\mathbb{R}}
\frac{|f''(t)|}{(f'(t))^2}I_{\{|f'|\ge \varepsilon\}}\, dt
\\
&+\|\varphi\|_\infty\varepsilon^{-1}\int_{\mathbb{R}}|\varrho'(t)|\, dt
+ \|\varphi'\|_\infty\int_{\mathbb{R}}I_{\{|f'|\le 2\varepsilon\}}|\varrho(t)|\, dt.
\end{align*}
We now observe that
$$
\frac{1}{2}
\int_{\mathbb{R}}
\frac{|f''(t)|}{(f'(t))^2}I_{\{|f'|\ge \varepsilon\}}\, dt
\le 
\int_{\mathbb{R}}
\frac{|f''(t)|}{(f'(t))^2+\varepsilon^2}\, dt.
$$
Furthermore, by the assumptions of the lemma,
\begin{align*}
\int_{\mathbb{R}}
\frac{|f''(t)|}{(f'(t))^2+\varepsilon^2}\, dt
&=
\sum_{j=1}^{r} \int_{a_{j-1}}^{a_j}
\frac{|f''(t)|}{(f'(t))^2+\varepsilon^2}\, dt
\\
&=
\sum_{j=1}^{r} \Bigl|\int_{a_{j-1}}^{a_j}
\frac{f''(t)}{(f'(t))^2+\varepsilon^2}\, dt\Bigr|
\le \pi r\varepsilon^{-1}.
\end{align*}
It can also be readily seen that
$$
\|\varrho\|_\infty \le \frac{1}{2}\int_{\mathbb{R}}|\varrho'(t)|\, dt = \frac{1}{2}\|\varrho'\|_{L^1(\mathbb{R})}.
$$
Summing up all the above estimates, we arrive at
\begin{align*}
\int_{\mathbb{R}}&\varphi'(f(t))\varrho(t)\, dt
\\
&\le 
\bigl((2\|\Phi'\|_\infty + 1)\pi r+1\bigr)\varepsilon^{-1}\|\varphi\|_\infty\|\varrho'\|_{L^1(\mathbb{R})}
+ \|\varphi'\|_\infty\int_{\mathbb{R}}I_{\{|f'|\le 2\varepsilon\}}\varrho(t)\, dt,
\end{align*}
which is exactly the claimed bound.
\end{proof}

\begin{corollary}\label{cor-key}
Let $\varrho\in C^\infty_0(\mathbb{R})$, $f\in C^\infty(\mathbb{R})$, 
$k\in \mathbb{N}$,
and assume that either $k\ge 3$ 
and $f^{(k)}(t)>0$ for all $t\in \mathbb{R}$ or $k=2$ and
$f''(t)\ge 0$ for all $t\in\mathbb{R}$.
Then, for every $\varphi \in C^\infty_b(\mathbb{R})$,
$$
\int_{\mathbb{R}}\varphi'(f(t))\varrho(t)\, dt
\le 
C(k-1)\varepsilon^{-1}\|\varphi\|_\infty\|\varrho'\|_{L^1(\mathbb{R})}
+ \|\varphi'\|_\infty\int_{\mathbb{R}}I_{\{|f'|\le 2\varepsilon\}}|\varrho(t)|\, dt
$$
for a universal constant $C>1$.
\end{corollary}

\begin{proof}
The assumption $f^{(k)}(t)>0$ for all $t\in\mathbb{R}$ with $k\ge 3$ implies that $f^{(k)}$ has no zeroes on $\mathbb{R}$. By Rolle's lemma, $f^{(k-1)}$ can then have at most one zero. Arguing inductively, we see that $f^{(k-j)}$ can have at most $j$ zeroes on $\mathbb{R}$. In particular, $f''$ has at most $k-2$ zeroes. This shows that the assumptions of the previous lemma are satisfied with $r=k-1$. 
In the case $k=2$, the assumptions of the previous lemma hold with $r=1$.
\end{proof}

\begin{remark}\label{rem-key}
Let $\varrho\in C^\infty_0(\mathbb{R})$ and let $f\in\mathcal{P}_d(\mathbb{R})$ with $d\ge 1$. Then, for every $\varphi \in C^\infty_b(\mathbb{R})$, 
$$
\int_{\mathbb{R}}\varphi'(f(t))\varrho(t)\, dt
\le 
Cd\varepsilon^{-1}\|\varphi\|_\infty\|\varrho'\|_{L^1(\mathbb{R})}
+ \|\varphi'\|_\infty\int_{\mathbb{R}}I_{\{|f'|\le 2\varepsilon\}}|\varrho(t)|\, dt
$$
for a universal constant $C>1$. Indeed, $f$ is of the form 
$$
f(t)=a_0+a_1t+\cdots+a_m t^m,\quad m\le d,
$$
with $a_m\ne 0$. Therefore, $f^{(m)}(t)=m!a_m\equiv \text{const}$.  

\noindent
If $m\ge 2$, one can apply Corollary~\ref{cor-key} to $f$ (or to $-f$ if $a_m<0$) with $k=m$.  

\noindent
If $m\in\{0,1\}$, then $f''(t)=0$, and once again Corollary~\ref{cor-key} applies to $f$ with $k=2$. In this case, $k-1=1\le d$.  
\end{remark}

To obtain a result on the fractional regularity 
of image measures in the one-dimensional setting, 
we will follow the argument of 
A.~Carbery, M.~Christ, and J.~Wright \cite{CCW99}. 
In particular, we will rely on the technical lemma below.

\begin{lemma}[see {\cite[Lemma 2.3]{CCW99}} or {\cite[Lemma 2.6.5]{Gr14}}]\label{CW-lem}
Let $a_0, \ldots, a_k$ be distinct real numbers.	
Let $a=\min\limits_{0\le j\le k} a_j$, $b=\max\limits_{0\le j\le k}a_j$, and let $f$ be a real-valued function such that $f \in C^{k-1}([a, b]) \cap C^k((a, b))$. Then there exists a point $y \in (a, b)$ such that
$$
f^{(k)}(y) = \sum_{m=0}^{k}c_mf(a_m),
$$
where $c_m=(-1)^kk!\prod\limits_{\substack{j=0\\ j\ne m}}^k(a_j-a_m)^{-1}$.
\end{lemma}

\begin{lemma}
Let $\mu$ be an absolutely continuous finite positive Borel measure on $\mathbb{R}$ with a bounded density $\varrho$, and let $E \subset \mathbb{R}$ be a measurable set of 
positive $\mu$-measure.
Then there exist points $a_0,\ldots,a_k \in E$ such that
$$
\|\varrho\|_\infty^k \prod\limits_{\substack{j=0\\ j\ne m}}^k|a_j-a_m|\ge \bigl(\mu(E)/4e\bigr)^k\quad \forall m\in\{0, \ldots, k\}.
$$
\end{lemma}

\begin{proof}
The proof follows the argument from Lemma 2.6.6 in  \cite{Gr14}.

Let $E'\subset E$ be a compact set such that $\mu(E\setminus E')<\mu(E)/2.$
Let 
$$
F(x):=\mu\bigl((-\infty, x]\cap E'\bigr).
$$
We observe that
$$
|F(x) - F(y)| = \Bigl|\int_{E'\cap (y, x]} \varrho(t)\, dt \Bigr|
\le \|\varrho\|_\infty|x-y|.
$$
In particular, the function $F$ is non-decreasing and continuous,
and therefore the 
mapping $F\colon E'\to [0, \mu(E')]$ is surjective.
Indeed, by the intermediate value theorem, $F$
is 
surjective as a mapping $F\colon[\inf E', \sup E']\to [0, \mu(E')]$. 
Let $y\in (0, \mu(E'))$ and suppose that
$F(x) = y$ for some $x\not \in E'$. Then 
$$
x\in [\inf E', \sup E']\setminus E' = \sqcup_j (a_j , b_j),
$$
where the intervals are disjoint. 
Hence, there exists a unique index $j_0$ such that 
$x\in (a_{j_0}, b_{j_0})$.
Since
$$ 
F(b_{j_0}) - F(a_{j_0})=
\mu \bigl((a_{j_0}, b_{j_0}]\cap E'\bigr) = 
\mu \bigl((a_{j_0}, b_{j_0})\cap E'\bigr) = 0,
$$
it follows that 
$F(a_{j_0})\le F(x)\le F(b_{j_0}) = F(a_{j_0})$. Since the intervals were disjoint,
we have $a_{j_0}\in E'$ and $F(a_{j_0}) = y$.

Now, let $a_j\in E'$ be chosen such that $F(a_j) = \frac{j}{k}\mu(E')$.
For this choice of points, we have
\begin{align*}
\|\varrho\|_\infty^k
\prod\limits_{\substack{j=0\\ j\ne m}}^k|a_j-a_m|
&\ge \prod\limits_{\substack{j=0\\ j\ne m}}^k\bigl|F(a_j)-F(a_m)\bigr|
=
\bigl(\mu(E')\bigr)^k
\prod\limits_{\substack{j=0\\ j\ne m}}^k\Bigl|\frac{j}{k}-\frac{m}{k}\Bigr|
\\
&=
\bigl(\mu(E')\bigr)^k\frac{(k-m)!m!}{k^k}
= \bigl(\mu(E')\bigr)^k\frac{k!}{k^kC_k^m}
\\
&\ge \bigl(\mu(E')\bigr)^k\frac{(k/e)^k}{k^k 2^k}
= \bigl(\mu(E')/2e\bigr)^k\ge \bigl(\mu(E)/4e\bigr)^k.
\end{align*}
The lemma is proved.
\end{proof}

\begin{corollary}\label{sub-level}
Let $\mu$ be an absolutely continuous 
finite positive Borel measure on $\mathbb{R}$
with a bounded density $\varrho$. 
Let $k\in \mathbb{N}$ and $f\in C^\infty(\mathbb{R})$ be such that
$f^{(k)}(t)\ge1$ for all $t\in \mathbb{R}$.
Then
$$
\mu(t\in\mathbb{R}\colon |f(t)|\le \varepsilon)\le 8ek\|\varrho\|_\infty\cdot \varepsilon^{1/k}\quad \forall \varepsilon>0.
$$ 
\end{corollary}

\begin{proof}
Let $E=\{t\in\mathbb{R}\colon |f(t)|\le \varepsilon\}$.
If $\mu(E)=0$ then the estimate is true.
Assume that $\mu(E)>0$. By the previous lemma,
there exist points $a_0, \ldots, a_k\in E$ such that
$$
\|\varrho\|_\infty^k\prod\limits_{\substack{j=0\\ j\ne m}}^k|a_j-a_m|\ge \bigl(\mu(E)/4e\bigr)^k\quad \forall m\in\{0, \ldots, k\}.
$$
On the other hand, by Lemma \ref{CW-lem},
$$
f^{(k)}(y)=(-1)^kk!\sum_{m=0}^{k}\prod\limits_{\substack{j=0\\ j\ne m}}^k(a_j-a_m)^{-1}f(a_m)
$$
for some $y\in \mathbb{R}$. Therefore,
$$
1\le f^{(k)}(y)\le k! \sum_{m=0}^{k}\prod\limits_{\substack{j=0\\ j\ne m}}^k|a_j-a_m|^{-1}\cdot |f(a_m)|\le (k+1)!(4e)^k\|\varrho\|_\infty^k\bigl(\mu(E)\bigr)^{-k}\varepsilon
$$
implying the estimate
$$
\mu(E)\le ((k+1)!)^{1/k}4e\|\varrho\|_\infty\cdot \varepsilon^{1/k}\le 8ek\|\varrho\|_\infty\cdot \varepsilon^{1/k}.	
$$
The corollary is proved.
\end{proof}

\begin{theorem}\label{T-1d-reg}
Let $\varrho\in BV(\mathbb{R})$, $k\in \mathbb{N}$, $k\ge 2$, and let $f\in C^\infty(\mathbb{R})$ be such that $f^{(k)}(t)\ge1$ for all $t\in\mathbb{R}$.
Then, for every $\varphi \in C^\infty_b(\mathbb{R})$,
$$
\int_{\mathbb{R}}\varphi'(f(t))\varrho(t)\, dt
\le 
Ck \|\varrho\|_\infty^{1-1/k}\|\varrho'\|_{\rm TV}^{1/k}
\|\varphi\|_\infty^{1/k}\|\varphi'\|_\infty^{1-1/k}
$$
for a universal constant $C>1$.
\end{theorem}

\begin{proof}
First, assume that $\varrho\in C^\infty_0(\mathbb{R})$.	
By Corollary \ref{cor-key},
$$
\int_{\mathbb{R}}\varphi'(f(t))\varrho(t)\, dt
\le 
C(k-1)\varepsilon^{-1}\|\varphi\|_\infty\|\varrho'\|_{L^1(\mathbb{R})}
+ \|\varphi'\|_\infty\int_{\mathbb{R}}I_{\{|f'|\le 2\varepsilon\}}|\varrho(t)|\, dt.
$$
Applying Corollary \ref{sub-level} with the measure $\mu(dt) = |\varrho(t)|\, dt$, we obtain 
$$
\int_{\mathbb{R}}I_{\{|f'|\le 2\varepsilon\}}|\varrho(t)|\, dt
\le 8e(k-1)\|\varrho\|_\infty\cdot (2\varepsilon)^{1/(k-1)}.
$$
Thus,
\begin{align*}
\int_{\mathbb{R}}&\varphi'(f(t))\varrho(t)\, dt
\\
&\le 
C(k-1)\varepsilon^{-1}\|\varphi\|_\infty\|\varrho'\|_{L^1(\mathbb{R})}
+ 8e2^{1/(k-1)}(k-1)\|\varphi'\|_\infty\|\varrho\|_\infty\cdot \varepsilon^{1/(k-1)},
\end{align*}
and by taking 
$$
\varepsilon =\Bigl( \frac{C\|\varphi\|_\infty\|\varrho'\|_{L^1(\mathbb{R})}}{8e2^{1/(k-1)}\|\varphi'\|_\infty\|\varrho\|_\infty}\Bigr)^{(k-1)/k},
$$
we obtain the estimate
\begin{align*}
\int_{\mathbb{R}}\varphi'(f(t))&\varrho(t)\, dt
\\
&\le 
2(k-1)(2C)^{1/k} (8e)^{1-1/k}
\|\varrho\|_\infty^{1-1/k}\|\varrho'\|_{L^1(\mathbb{R})}^{1/k}
\|\varphi\|_\infty^{1/k}\|\varphi'\|_\infty^{1-1/k}
\\
&\le
100Ck
\|\varrho\|_\infty^{1-1/k}\|\varrho'\|_{L^1(\mathbb{R})}^{1/k}
\|\varphi\|_\infty^{1/k}\|\varphi'\|_\infty^{1-1/k}.
\end{align*}

Let now $\varrho\in BV(\mathbb{R})$.
Let $\eta, \omega\in C_0^\infty(\mathbb{R})$ be a pair of non-negative functions such that $\eta(t)=1$ $\forall t\in[-1, 1]$,
$\eta(t)=0$ $\forall t\in\mathbb{R}\setminus[-2, 2]$,
$\eta(t)\in[0, 1]$ $\forall t\in \mathbb{R}$, and 
$\int_\mathbb{R}\omega(t)\, dt = 1$.
Let 
$$
\omega_m(s):=m\omega(ms)
\hbox{ and } 
\varrho_{n ,m}(t): =
\eta(n^{-1}t)\cdot \varrho*\omega_m(t)
\in C_0^\infty(\mathbb{R}).
$$
For every $\varphi \in C^\infty_b(\mathbb{R})$,
we have
\begin{align}\label{eq-mid}
\int_{\mathbb{R}}\varphi'(f(t))\varrho_{n, m}(t)&\, dt
\\
&\le 
100Ck \|\varrho_{n, m}\|_\infty^{1-1/k}\|\varrho_{n, m}'\|_{L^1(\mathbb{R})}^{1/k}
\|\varphi\|_\infty^{1/k}\|\varphi'\|_\infty^{1-1/k}.\nonumber
\end{align}
We note that 
$$
\|\varrho_{n, m}\|_\infty\le 
\|\varrho*\omega_m\|_\infty
\le \|\varrho\|_\infty \|\omega_m\|_{L^1(\mathbb{R})}
= \|\varrho\|_\infty
$$
and
$$
\|\varrho_{n, m}'\|_{L^1(\mathbb{R})}\le n^{-1}\|\eta'\|_\infty\|\varrho\|_{L^1(\mathbb{R})}
+\|\varrho'\|_{\rm TV}.
$$
In addition,
$$
\varrho_{n, m}\xrightarrow[n\to \infty]{L^1(\mathbb{R})}
\varrho*\omega_m
\hbox{ and }
\varrho*\omega_m 
\xrightarrow[m\to \infty]{L^1(\mathbb{R})}\varrho
$$
by the standard properties of convolution.
Now, passing to the limits in \eqref{eq-mid},
we obtain the announced bound.
\end{proof}

\begin{remark}
Now, to deduce Theorem~\ref{T-1d}, it is sufficient to observe that  
$$
\|\varrho\|_\infty\le \|\varrho'\|_{\rm TV}
$$
for any function $\varrho \in BV(\mathbb{R})$.
\end{remark}

\section{Polynomial images of $s$-concave measures}\label{sec-mult-dim}

We begin with the following key corollary of the one-dimensional estimate.

\begin{corollary}\label{cor-pol}	
There exists a universal constant $C>1$
such that for all
$d, n\in\mathbb{N}$, $f\in \mathcal{P}_d(\mathbb{R}^n)$, $\varrho\in BV(\mathbb{R}^n)$,
$\theta\in \mathbb{R}^n$ with $|\theta|=1$, and $\varphi \in C^\infty_b(\mathbb{R})$,
one has
$$
\int_{\mathbb{R}^n}\varphi'(f(x))\varrho(x)\, dx
\le 
Cd\|\varphi\|_\infty\varepsilon^{-1}\|D_\theta\varrho\|_{\rm TV}
+ \|\varphi'\|_\infty\int_{\mathbb{R}^n}I_{\{ |\partial_\theta f|\le 2\varepsilon\}} |\varrho(x)|\, dx.
$$
\end{corollary}

\begin{proof}
First, assume that $\varrho\in C_0^\infty(\mathbb{R}^n)$.
Let $\theta\in \mathbb{R}^n$ be a fixed unit vector. 
For any fixed $y\in \langle \theta\rangle^\bot$,
the function $g_y(t)= f(y+t\theta)$ is a polynomial of degree at most $d$.
Therefore, by Corollary~\ref{cor-key} and Remark~\ref{rem-key},
we have
\begin{align*}
\int_{\mathbb{R}^n}\varphi'(f(x))\varrho(&x)\, dx
=
\int_{\langle\theta \rangle^\bot} \int_\mathbb{R}
\varphi'(f(y+t\theta)) \varrho(y+t\theta)\, dt\, dy
\\
&\le Cd\|\varphi\|_\infty\varepsilon^{-1}
\int_{\langle\theta \rangle^\bot}
\int_{\mathbb{R}}|\partial_\theta\varrho(y+t\theta)|\, dt\, dy
\\
&+ \|\varphi'\|_\infty\int_{\langle\theta \rangle^\bot}\int_{\mathbb{R}}I_{\{|\partial_\theta f|\le 2\varepsilon\}}|\varrho(y+t\theta)|\, dt\, dy
\\
&= Cd\|\varphi\|_\infty\varepsilon^{-1}
\|\partial_\theta\varrho\|_{L^1(\mathbb{R}^n)}
+ \|\varphi'\|_\infty \int_{\mathbb{R}^n}I_{\{ |\partial_\theta f|\le 2\varepsilon\}} |\varrho(x)|\, dx.
\end{align*}

Let now $\varrho\in BV(\mathbb{R}^n)$.
Let $\eta, \omega\in C_0^\infty(\mathbb{R}^n)$ be a pair of non-negative functions such that $\eta(x)=1$ if $|x|\le 1$,
$\eta(x)=0$ if $|x|\ge 2$,
$\eta(x)\in[0, 1]$ $\forall x\in \mathbb{R}^n$, and 
$\int_{\mathbb{R}^n}\omega(x)\, dx = 1$.
Let $\omega_m(x):= m^n\omega(mx)$ and
$$
\varrho_{k ,m}(x): = \eta(k^{-1}x)\cdot \varrho*\omega_m(x)\in C_0^\infty(\mathbb{R}^n).
$$
By the standard properties of convolution, 
$$
\varrho_{k, m}\xrightarrow[k\to\infty]{L^1(\mathbb{R}^n)}\varrho*\omega_m
\hbox{ and } \varrho*\omega_m\xrightarrow[n\to\infty]{L^1(\mathbb{R}^n)} \varrho.
$$
In addition,
$$
\|\varrho_{k, m}\|_\infty\le \|\varrho\|_\infty
$$ 
and
\begin{align*}
\|\partial_\theta\varrho_{k, m}\|_{L^1(\mathbb{R}^ n)}
&= \|k^{-1}\partial_\theta\eta(k^{-1}\cdot)\varrho*\omega_m + \eta(k^{-1}\cdot)\varrho*(\partial_\theta\omega_m)\|_{L^1(\mathbb{R}^ n)}
\\
&\le k^{-1}\|\partial_\theta \eta\|_\infty\|\varrho\|_{L^1(\mathbb{R}^n)}+\|D_\theta\varrho\|_{\rm TV},	
\end{align*}
where the last bound follows from the
estimate 
\begin{align*}
\|\varrho*(\partial_\theta\omega_m)&\|_{L^1(\mathbb{R})}
=\sup\limits_{\substack{u\in C_0^\infty(\mathbb{R}^n)\\ \|u\|_\infty\le1}}\int_{\mathbb{R}^n}u(x)\int_{\mathbb{R}^n}\varrho(y)\partial_\theta\omega_m(x-y)\, dy\, dx
\\
& = 
\sup\limits_{\substack{u\in C_0^\infty(\mathbb{R}^n)\\ \|u\|_\infty\le1}}\int_{\mathbb{R}^n}u(x)\int_{\mathbb{R}^n}
\omega_m(x-y)\, D_\theta\varrho(dy)\, dx
\\
&=
\sup\limits_{\substack{u\in C_0^\infty(\mathbb{R}^n)\\ \|u\|_\infty\le1}}\int_{\mathbb{R}^n}\Bigl(\int_{\mathbb{R}^n}u(x)
\omega_m(x-y)\, dx\Bigr) \, D_\theta\varrho(dy)
\le \|D_\theta\varrho\|_{\rm TV}.
\end{align*}
For the function $\varrho_{k, m}$, we have already proved that
\begin{align*}
\int_{\mathbb{R}^n}\varphi'(f(x))\varrho_{k, m}(x)\, dx
&\le Cd\|\varphi\|_\infty\varepsilon^{-1}
(k^{-1}\|\partial_\theta \eta\|_\infty\|\varrho\|_{L^1(\mathbb{R}^n)}+\|D_\theta\varrho\|_{\rm TV})
\\
&+ \|\varphi'\|_\infty \int_{\mathbb{R}^n}I_{\{ |\partial_\theta f|\le 2\varepsilon\}} |\varrho_{k, m}(x)|\, dx.
\end{align*}
T aking the limits, first as $k\to \infty$
and then as $m\to \infty$, we obtain the desired estimate.
\end{proof}

The following corollary extends the sub-level estimate~\eqref{CW-est} 
to the class of all $s$-concave measures with $s \ge 0$.

\begin{corollary}\label{CW-cor}
There exists a constant $C>0$ such that
for every pair of integers $n, d\in \mathbb{N}$, for every $s\in [0, 1/n]$, for every $s$-concave measure $\mu$ on $\mathbb{R}^n$, and for every polynomial $f\in\mathcal{P}_d(\mathbb{R}^n)$, one has
\begin{equation}\label{CW-conc}
\|f\|_{L^2(\mu)}^{1/d}\mu(|f|\le \varepsilon)
\le C\min(d, 1/s) \varepsilon^{1/d}\quad \forall \varepsilon>0.
\end{equation}
\end{corollary}

\begin{proof}
If $\mu$ is a Dirac measure, then the estimate \eqref{CW-conc} holds with any constant $C \ge 1$.  
We now assume that $\mu$ is not a Dirac measure. 
Recall that $S(\mu)$ denotes the affine subspace spanned by the support of $\mu$.
In this case, $\dim S(\mu) \in \{1, \ldots, n\}$, and $\mu$ is absolutely continuous with respect to the Lebesgue measure on $S(\mu)$. 
Since the restriction of a polynomial of degree at most $d$ to any affine subspace is again a polynomial of degree at most $d$, it suffices to prove the estimate \eqref{CW-conc} for absolutely continuous measures $\mu$  (i.e., in the case $S(\mu) = \mathbb{R}^n$).  
Let $\varrho$ denote its density. By Theorem~\ref{T-Bor}, $\varrho$ is $\gamma$-concave with $\gamma = \frac{s}{1 - s n} \in [0, +\infty]$.
Assume first that $s > 0$, and take $m \in \mathbb{N}$ such that
$m-1\le 1/\gamma< m$.
Let
$$
K := \bigl\{ (x, y) \in \mathbb{R}^n \times \mathbb{R}^m \colon 
\varrho(x) > 0, \, 
|y| \le \kappa_m^{-1/m} \varrho(x)^{1/m} \bigr\},
$$
where $\kappa_m$ is the volume of the unit ball in $\mathbb{R}^m$.
Since $\varrho^\gamma$ is concave on the set $\{x \in \mathbb{R}^n \colon \varrho(x)>0\}$ and $1/(m\gamma) \le 1$,  
the function $\varrho^{1/m} = (\varrho^\gamma)^{1/(m\gamma)}$ is also concave
on $\{x \in \mathbb{R}^n \colon \varrho(x)>0\}$,  
which implies that the set $K$ is a convex subset of $\mathbb{R}^{n+m}$.  
In addition, the set $\{x \in \mathbb{R}^n \colon \varrho(x)>0\}$ is bounded due to \cite[Remark~2.2.7~(i)]{BGVV},  
which implies that the set $K$ is also bounded.
Moreover, for any bounded measurable function
$u$ on $\mathbb{R}^n$, we have
$$
\int_{\mathbb{R}^n}u(x)\varrho(x)\, dx =
\int_Ku(x)\, dxdy.
$$
In particular, $\lambda_{n+m}(K) = 1$, and
for any polynomial $f$ of degree at most $d$
on $\mathbb{R}^n$, we have
$$
\mu\bigl(x\in \mathbb{R}^n\colon |f|\le \varepsilon\bigr)
= \lambda_{n+m}\bigl((x, y)\in K\colon |f(x)|\le \varepsilon\bigr)
$$
and
$$
\int_{\mathbb{R}^n}|f(x)|^2\, \mu(dx) = \int_{K}|f(x)|^2\, dxdy.
$$
Applying the bound~\eqref{CW-est}, we arrive at the estimate
\begin{align*}
\|f\|_{L^2(\mu)}^{1/d}\mu(|f|\le \varepsilon)
&\le C\min(d, n+m) \varepsilon^{1/d}
\\
&\le C\min(d, 1/s +1) \varepsilon^{1/d}
\le 2C\min(d, 1/s) \varepsilon^{1/d}.
\end{align*}
For $s = 0$, see \cite{NSV}. Alternatively, one can use the fact that any log-concave measure can be obtained as a weak limit of marginals of high-dimensional uniform distributions on convex sets 
(see \cite[Remark 2.2.7 (ii)]{BGVV} or \cite[Theorem 9.1.6]{A-AGM-2}).
\end{proof}	

The following lemma is a dimension-dependent regularity result for polynomial images of high-dimensional $s$-concave measures.

\begin{lemma}\label{lem-dim}
Let $n \in \mathbb{N}$. There exists a constant $C(n)>0$, depending only on $n$, such that for any $s \in [0, 1/n]$, any absolutely continuous $s$-concave measure $\mu$ on $\mathbb{R}^n$, any $f \in \mathcal{P}_d(\mathbb{R}^n)$, and any $\varphi \in C_b^\infty(\mathbb{R})$, one has \begin{align}\label{eq-lem}
\|f-\mathbb{E}_\mu f\|_{L^2(\mu)}^{1/d}
\int_{\mathbb{R}^n}\varphi'(f(x))&\, \mu(dx)
\\
&\le 
C(n)\min\{d, 1/s\}\|\varphi\|_\infty^{1/d}\|\varphi'\|_\infty^{1-1/d},\nonumber
\end{align}
where $\displaystyle\mathbb{E}_\mu f:=\int_{\mathbb{R}^n} f\, d\mu$.
\end{lemma}

\begin{proof}
	
Let $\varrho$ be the density of the measure $\mu$. 
For any non-degenerate linear transformation $T\colon\mathbb{R}^n\to\mathbb{R}^n$ and any $h\in\mathbb{R}^n$, the function 
$|\det T|\cdot\varrho(Tx+h)$ is again the density of an $s$-concave measure, $f(Tx+h)$ is also a polynomial of degree at most $d$, and
$$
\int_{\mathbb{R}^n}u(f(x))\varrho(x)\, dx
=\int_{\mathbb{R}^n}u\bigl(f(Tx+h)\bigr)
|\det T| \cdot\varrho(Tx+h)\, dx
$$
for any bounded measurable function $u$. Therefore, it suffices to prove the estimate~\eqref{eq-lem} only for some affine image of the measure $\mu$.
Since $s\ge0$, the measure $\mu$ is also log-concave, and there exist a non-degenerate linear mapping 
$T\colon\mathbb{R}^n\to\mathbb{R}^n$ 
and a shift $h\in\mathbb{R}^n$ such that 
$|\det T|\cdot\varrho(Tx+h)$
is the density of an isotropic measure (see \cite[Section~2.3.3]{BGVV} or \cite[Section~10.2]{A-AGM-1}). Therefore, without loss of generality, we may assume that the measure $\mu$ in~\eqref{eq-lem} is isotropic.
\vskip .05in
\noindent	
{\bf Step 1.} Let now $\mu$ be isotropic.
By Corollary~\ref{cor-pol},
we know that
$$
\int_{\mathbb{R}^n}\varphi'(f(x))\, \mu(dx)
\le 
C_1d\varepsilon^{-1}\|\varphi\|_\infty\|D_\theta\varrho\|_{\rm TV}
+ \|\varphi'\|_\infty\mu\bigl(|\partial_\theta f|\le 2\varepsilon\bigr).
$$
Suppose $\|\partial_\theta f\|_{L^2(\mu)}>0$. 
As $\partial_\theta f$ is a polynomial of degree at most $d-1$, 
Corollary~\ref{CW-cor} yields
\begin{align*}
\int_{\mathbb{R}^n}&\varphi'(f(x))\, \mu(dx)
\\
&\le 
C_1d\varepsilon^{-1}\|\varphi\|_\infty\|D_\theta\varrho\|_{\rm TV}
+ C_2\min(d, 1/s)\varepsilon^{1/(d-1)} \|\varphi'\|_\infty  \|\partial_\theta f\|_{L^2(\mu)}^{-1/(d-1)}.
\end{align*} 
By taking
$$
\varepsilon = \bigl(d\|\varphi\|_\infty\|D_\theta\varrho\|_{\rm TV} (\min(d, 1/s))^{-1} \|\varphi'\|_\infty^{-1}  \|\partial_\theta f\|_{L^2(\mu)}^{1/(d-1)}\bigr)^{(d-1)/d},
$$
we obtain
\begin{align*}
\int_{\mathbb{R}^n}&\varphi'(f(x))\, \mu(dx)
\\
&\le 
(C_1+C_2)d^{1/d}(\min(d, 1/s))^{1-1/d}
\|\varphi\|_\infty^{1/d}\|\varphi'\|_\infty^{1-1/d}
\|D_\theta\varrho\|_{\rm TV}^{1/d}  \|\partial_\theta f\|_{L^2(\mu)}^{-1/d}.
\end{align*}
Noting that $d^{1/d}\le 2$ and that $\min(d,1/s)\ge 1$, we in fact have
$$
\|\partial_\theta f\|_{L^2(\mu)}^{1/d}\int_{\mathbb{R}^n}\varphi'(f(x))\, \mu(dx)
\le 
C_3\min(d, 1/s)
\|\varphi\|_\infty^{1/d}\|\varphi'\|_\infty^{1-1/d}
\|D_\theta\varrho\|_{\rm TV}^{1/d}
$$
and this estimate remains valid also in the case  	 
$\|\partial_\theta f\|_{L^2(\mu)}=0$.
\vskip .05in
\noindent	
{\bf Step 2.}
By \eqref{Krug}, the norm of the derivative is given by$$
\|D_\theta \varrho\|_{\rm TV} = 
2\cdot \int_{\langle\theta\rangle^\bot}
\sup_{t\in\mathbb{R}}\varrho(y+t\theta)\, dy.
$$ 
Since we have assumed that $\mu$ is isotropic, 
by Theorem~\ref{t-Klar} we obtain
$$
\varrho(x)\le \varrho(0)e^{C_4n-C_5|x|},
$$
which implies
\begin{align*}
\|D_\theta \varrho\|_{\rm TV} 
&\le
2e^{C_4n}\varrho(0)\cdot 
\int_{\langle\theta\rangle^\bot}
e^{-C_5|y|}\, dy
\\
&=C_6(n)\varrho(0) \le 
C_6(n)\|\varrho\|_\infty =
C_6(n)L(\mu)^n\le C_7(n),
\end{align*}
where we have used the dimensional estimate \eqref{isotr-est} for the isotropic constant $L(\mu)$.
Thus, we arrive at the bound
$$
\|\partial_\theta f\|_{L^2(\mu)}^{1/d}\int_{\mathbb{R}^n}\varphi'(f(x))\, \mu(dx)
\le 
C_8(n)\min(d, 1/s)
\|\varphi\|_\infty^{1/d}\|\varphi'\|_\infty^{1-1/d}.
$$
\vskip .05in
\noindent	
{\bf Step 3.}
Without loss of generality, we may assume that
$$
\int_{\mathbb{R}^n}\varphi'(f(x))\, \mu(dx)>0.
$$
Let $e \in \mathbb{R}^n$ be any unit vector, for instance, $e = (1, 0, \dots, 0)$.
Integrating the $2d$-th power of the above estimate  
\begin{align*}
\Bigl(\int_{\mathbb{R}^n}|\langle\nabla f, \theta\rangle|^2\, d\mu\Bigr)
&\Bigl(\int_{\mathbb{R}^n}\varphi'(f(x))\, \mu(dx)\Bigr)^{2d}
\\
&\le \Bigl(C_8(n)\min(d, 1/s)
\|\varphi\|_\infty^{1/d}\|\varphi'\|_\infty^{1-1/d}\Bigr)^{2d}
\end{align*}
over the surface measure on the unit sphere of $\mathbb{R}^n$, we obtain
\begin{align*}
&\Bigl(\int_{\mathbb{R}^n} |\nabla f|^2\, d\mu\Bigr)
\Bigl(\int_{S^{n-1}}|\langle e, \theta\rangle|^2\, \sigma_{n-1}(d\theta)\Bigr)
\Bigl(\int_{\mathbb{R}^n}\varphi'(f(x))\, \mu(dx)\Bigr)^{2d}
\\
&=	
\Bigl(\int_{S^{n-1}}\Bigl(\int_{\mathbb{R}^n}|\langle\nabla f, \theta\rangle|^2\, d\mu\Bigr)\, \sigma_{n-1}(d\theta)\Bigr)
\Bigl(\int_{\mathbb{R}^n}\varphi'(f(x))\, \mu(dx)\Bigr)^{2d}
\\
&\le 
n\kappa_n\Bigl(C_8(n)\min(d, 1/s)
\|\varphi\|_\infty^{1/d}\|\varphi'\|_\infty^{1-1/d}\Bigr)^{2d}.
\end{align*}
Finally, applying the Poincar\'e inequality 
for isotropic log-concave measures from Theorem~\ref{Poinc}, 
we arrive at
\begin{align*}
\Bigl(\int_{\mathbb{R}^n}\bigl(f-\mathbb{E}_\mu f\bigr)^2\, d\mu\Bigr)
&\Bigl(\int_{\mathbb{R}^n}\varphi'(f(x))\, \mu(dx)\Bigr)^{2d}
\\
&\le C_9(n)\Bigl(C_8(n)\min(d, 1/s)
\|\varphi\|_\infty^{1/d}\|\varphi'\|_\infty^{1-1/d}\Bigr)^{2d},
\end{align*}
which is equivalent to the estimate stated in the lemma.
\end{proof}

Finally, applying the localization technique from Theorem~\ref{loc-lem}, we deduce a dimension-independent bound from the dimension-dependent result of Lemma~\ref{lem-dim}.

\begin{theorem}\label{CW-T}
There exists an absolute constant $C>0$ such that
for every $s\ge 0$, every $d, n\in\mathbb{N}$, $n\le 1/s$, every $s$-concave measure $\mu$ on $\mathbb{R}^n$,
every $f\in \mathcal{P}_d(\mathbb{R}^n)$, and every $\varphi\in C_b^\infty(\mathbb{R})$, one has
$$
\|f-\mathbb{E}_\mu f\|_{L^2(\mu)}^{1/d}
\int_{\mathbb{R}^n}\varphi'(f(x))\, \mu(dx)
\le 
C\min\{d, 1/s\}\|\varphi\|_\infty^{1/d}\|\varphi'\|_\infty^{1-1/d},
$$
where $\displaystyle\mathbb{E}_\mu f:=\int_{\mathbb{R}^n} f\, d\mu$.	

\noindent
In other words, for every 
$f\in \mathcal{P}_d(\mathbb{R}^n)$
with
$\|f-\mathbb{E}_\mu f\|_{L^2(\mu)}>0,$
one has
$$
\sigma(\mu\circ f^{-1}, t)
\le 
\frac{C\min\{d, 1/s\}} {\|f-\mathbb{E}_\mu f\|_{L^2(\mu)}^{1/d}}\cdot t^{1/d}\quad \forall t>0.
$$
\end{theorem}	

\begin{proof}	
We prove the announced bound with $C=\max\{C(1), C(2), C(3)\}$,
where the constants $C(1)$, $C(2)$, and $C(3)$ 
are from  Lemma \ref{lem-dim}.
Without loss of generality, we may assume that
$$
\|f-\mathbb{E}f\|_{L^2(\mu)}>0.
$$	
In particular, $\mu$ is not a Dirac measure.
If $\dim S(\mu)\in\{1, 2, 3\}$, then, by Lemma~\ref{lem-dim},
\begin{align*}
\|f-\mathbb{E}f\|_{L^2(\mu)}^{1/d}
\int_{\mathbb{R}^n}&\varphi'(f(x))\, \mu(dx)
\\
&\le 
\max\{C(1), C(2), C(3)\}\min\{d, 1/s\}\|\varphi\|_\infty^{1/d}\|\varphi'\|_\infty^{1-1/d}.
\end{align*}
Thus, it suffices to prove the estimate only for $s$-concave measures $\mu$ satisfying $\dim S(\mu)\ge 4$.  
Therefore, without loss of generality, we assume that $s \in [0, 1/4]$ and $n \in [4, 1/s]$.
We now fix a number $a>0$,
a polynomial $g\in \mathcal{P}_d(\mathbb{R}^n)$, a function $\psi\in C_b^\infty(\mathbb{R})$, 
and a convex compact subset $K\subset\mathbb{R}^n$.
We consider the set $P^s_{u_1, u_2, u_3}(K)$
of all $s$-concave measures $\nu$ supported in $K$ such that
$$
\int_{\mathbb{R}^n} u_1\, d\nu\ge 0, \quad
\int_{\mathbb{R}^n} u_2\, d\nu\ge 0, \quad
\int_{\mathbb{R}^n} u_3\, d\nu\ge 0,
$$
where $u_1=g$, $u_2=-g$, $u_3 = g^2-a$,
i.e.,
$$
\mathbb{E}g=\int_{\mathbb{R}^n}g\, d\nu = 0
\hbox{ and }
\int_{\mathbb{R}^n}(g-\mathbb{E}g)^2\, d\nu = 
\int_{\mathbb{R}^n}g^2\, d\nu \ge a\quad\forall
\nu \in P^s_{u_1, u_2, u_3}(K).
$$
Let
$$
F(\nu) = \int_{\mathbb{R}^n}\psi'(g)\, d\nu.
$$
This functional is linear and continuous in the weak topology on the space of measures supported in $K$.
By the localization lemma (Theorem~\ref{loc-lem})
and Lemma \ref{lem-dim},
\begin{align*}
\sup\limits_{\nu\in P_{u_1, u_2, u_3}} F&(\nu)
\le \sup\limits_{\substack{\nu\in P_{u_1, u_2, u_3} \\ \dim S(\nu)\le 3}} F(\nu)
\\
&\le \max\{C(1), C(2), C(3)\}a^{-1/2d}
\min\{d, 1/s\}\|\psi\|_\infty^{1/d}\|\psi'\|_\infty^{1-1/d}.	
\end{align*}
Let now $\mu$ be any $s$-concave measure on 
$\mathbb{R}^n$ with $s\in(0, 1/4]$ and $n\ge 4$,
let $f\in \mathcal{P}_d(\mathbb{R}^n)$, and let $\varphi\in C_b^\infty(\mathbb{R})$.
Since we have assumed that $s>0$, the convex set
$\supp(\mu)$ is compact (see \cite{Bobk07} or \cite[Remark~2.2.7~(i)]{BGVV}). Let $K = \supp(\mu)$,
$g = f - \mathbb{E}_\mu f$,
$a = \|f - \mathbb{E}_\mu f\|_{L^2(\mu)}^2$,
$\psi(t) = \varphi(t + \mathbb{E}_\mu f)$.
Then 
$\mu\in P^s_{u_1, u_2, u_3}(K)$
with $u_1=g$, $u_2=-g$, $u_3 = g^2-a$,
and, as we have already shown above,
\begin{align*}
\int_{\mathbb{R}^n}&\varphi'(f)\, d\mu
=
\int_{\mathbb{R}^n}\psi'(g)\, d\mu=
F(\mu)
\\
&\le 
\max\{C(1), C(2), C(3)\}a^{-1/2d}
\min\{d, 1/s\}\|\psi\|_\infty^{1/d}\|\psi'\|_\infty^{1-1/d}
\\
&= 
\max\{C(1), C(2), C(3)\}\|f - \mathbb{E}f\|_{L^2(\mu)}^{-1/d}
\min\{d, 1/s\}\|\varphi\|_\infty^{1/d}\|\varphi'\|_\infty^{1-1/d},
\end{align*}
as announced. 

Finally, the estimate in the log-concave case $s = 0$ follows from the estimate already obtained for $s>0$ and from \cite[Remark~2.2.7~(ii)]{BGVV}.  
This completes the proof of the theorem.
\end{proof}

\begin{remark}
Theorem~\ref{T-main-reg} now follows from the theorem above, since the uniform distribution on a convex set is a $1/n$-concave measure by the Brunn--Minkowski inequality.
\end{remark}

\begin{remark}\label{rem-CW-T}
In probabilistic terms, the above theorem can be stated as follows:	

\noindent	
{\it	
There exists an absolute constant $C>0$ such that,	
for every $s\ge 0$, every $d, n\in\mathbb{N}$, $n\le 1/s$, every $n$-dimensional random vector
$X$ with an $s$-concave distribution,
every $f\in \mathcal{P}_d(\mathbb{R}^n)$, and every $\varphi\in C_b^\infty(\mathbb{R})$, one has
$$
\bigl(\mathbb{D}(f(X))\bigr)^{1/(2d)}
\mathbb{E}\bigl(\varphi'(f(X))\bigr)
\le 
C\min\{d, 1/s\}\|\varphi\|_\infty^{1/d}\|\varphi'\|_\infty^{1-1/d},
$$
where $\mathbb{E}$ and $\mathbb{D}$ denote expectation and variance of random variables, respectively.	
}
\end{remark}

\section{Polynomial images of the uniform distribution on the unit cube}\label{sec-cube}

Throughout this section, 
let $Q^n:= [-\frac{1}{2}, \frac{1}{2}]^n$
and let $\mu_{Q^n}$ denote the uniform probability measure on $Q^n$.

\subsection{General polynomials of fixed total degree}

We start with a regularity counterpart of 
the estimate \eqref{Th-CCW-1}.

\begin{proposition}
There exists an absolute constant 
$C>1$ such that
for every $n,k,d\in\mathbb{N}$, every choice of $k_1,\dots,k_n\in\mathbb{N}\cup\{0\}$ with $k=k_1+\cdots+k_n$, and every
$f\in\mathcal{P}_d(\mathbb{R}^n)$ satisfy
$$
\Bigl|\frac{\partial^k}{\partial x_1^{k_1}\ldots \partial x_n^{k_n}}f(x)\Bigr|\ge 1\quad \forall x\in Q^n,
$$
one has
$$
\sigma(\mu_{Q^n}\circ f^{-1}, t)
\le 
Cdk t^{1/k}.
$$
\end{proposition}

\begin{proof}	
From Corollary \ref{cor-pol} we know that
$$
\int_{Q^n}\varphi'(f(x))\, dx
\le 
cd\|\varphi\|_\infty\varepsilon^{-1}\|D_{e_j} I_{Q^n}\|_{\rm TV}
+ \|\varphi'\|_\infty\int_{Q^n}I_{\{ |\frac{\partial f}{\partial x_j}|\le 2\varepsilon\}}\, dx
$$
for every $\varphi\in C_b^\infty(\mathbb{R})$,
where $e_j$ denotes the $j$-th standard basis vector in~$\mathbb{R}^n$.
By \eqref{Krug}, we have
$\|D_{e_j} I_{Q^n}\|_{\rm TV}=2$.
Therefore,
\begin{equation}\label{eq-cube}
\int_{Q^n}\varphi'(f(x))\, dx
\le 
2cd\|\varphi\|_\infty\varepsilon^{-1}
+ \|\varphi'\|_\infty\mu_{Q^n}\bigl(|\partial f/\partial x_j|\le 2\varepsilon\bigr).
\end{equation}

We proceed by induction. The base case $k=1$,
i.e. $\bigl|\frac{\partial f}{\partial x_j}\bigr|\ge 1$ on $Q^n$, is immediate.
In this case, taking $\varepsilon=1/3$, we obtain the announced estimate.

Inductive step. Assume that $k\ge 2$ and that $k_j\ge1$ for some $j$. 
By the inductive hypothesis, we obtain that
$$
\sigma(\mu_{Q^n}\circ (\partial f/\partial x_j)^{-1}, t)
\le 
Cd(k-1) t^{1/(k-1)}.
$$
In particular, by Theorem \ref{T-meas}, we have
$$
\mu_{Q^n}\bigr(|\partial f/\partial x_j|\le 2\varepsilon\bigl)
\le \sigma(\mu_{Q^n}\circ (\partial f/\partial x_j)^{-1}, 4\varepsilon)\le Cd(k-1) (4\varepsilon)^{1/(k-1)},
$$
and \eqref{eq-cube} then implies
$$
\sigma(\mu_{Q^n}\circ f^{-1}, t)
\le 2cdt\varepsilon^{-1} + Cd(k-1) (4\varepsilon)^{1/(k-1)}.
$$
By taking $\varepsilon = 4^{-1/k} (2c/ C)^{(k-1)/k} t^{(k-1)/k}$, we arrive at the estimate
\begin{align*}
\sigma(\mu_{Q^n}\circ f^{-1}, t)
&\le \bigl((8c)^{1/k}C^{(k-1)/k} + (k-1) (8c)^{1/k}C^{(k-1)/k} \bigr)dt^{1/k} 
\\
&= (8c)^{1/k}C^{(k-1)/k}dkt^{1/k}.
\end{align*}
Thus, the conclusion of the proposition holds with $C = 8c$.
\end{proof}

The following statement is a special case of Theorem~\ref{reg-prod-1}, corresponding to the unit cube.

\begin{proposition}\label{cor-coeff}
There exists an absolute constant $C>0$ such that
for all $d, n\in \mathbb{N}$ and for every 
non-constant $f\in \mathcal{P}_d(\mathbb{R}^n)$,
one has
$$
\sigma(\mu_{Q^n}\circ f^{-1}, t)
\le 
C \min\{d, n\} [f]_2^{-1/d} t^{1/d}\quad \forall t>0.
$$
\end{proposition}

\begin{proof}
By Theorem \ref{CW-T}, we have
$$
\sigma(\mu_{Q^n}\circ f^{-1}, t)
\le 
C\min\{d, n\} \|f-\mathbb{E}f\|_{L^2(Q^n)}^{-1/d} t^{1/d}\quad \forall t>0.
$$
It follows from~\cite[Theorem~1(2)]{GM22} that
$$
\|f-\mathbb{E}f\|_{L^2(Q^n)}\ge c^d [f]_2
$$
for some universal constant $c>0$.
Combining these two estimates, we obtain the claimed bound.
\end{proof}

\subsection{Polynomials with bounded individual degree}

We recall that $\mathcal{P}_{d,m}(\mathbb{R}^n)$ stands for
the space of all polynomials of total degree at most $d$ and individual degree at most $m$ (see Definition~\ref{def-pol}).

\begin{theorem}\label{ind-deg}
For $d, m\in \mathbb{N}$ with $d\ge m$, there exists a constant $C(m, d)$,
depending only on these parameters, such that
for every non-constant polynomial $f\in \mathcal{P}_{d,m}(\mathbb{R}^n)$ (i.e. $d(f)\ge1$),
one has
$$
\sigma(\mu_{Q^n}\circ f^{-1}, t)
\le 
C(m, d) [f]_\infty^{-1/m} t^{1/m}\bigl(|\ln ([f]_\infty^{-1}t)|^{d-m}+1\bigr)
\quad \forall t>0.
$$
\end{theorem}

\begin{proof}
Let $k_1,\ldots,k_n$ with $k_1+\cdots+k_n = d(f)$ be such that
$$
[f]_\infty:=\max\{|a_{j_1,\ldots, j_n}|\colon j_1+\ldots+j_n= d(f)\} = |a_{k_1, \ldots, k_n}|.
$$
Without loss of generality, we may assume that $[f]_\infty = 1$.
For brevity, we will use probabilistic notation.
Let $U=(U_1,\ldots,U_n)$ be a random vector with independent coordinates uniformly distributed on $\bigl[-\tfrac{1}{2},\tfrac{1}{2}\bigr]$. Then $\mu_{Q^n}$ coincides with the distribution of $U$.

\vskip .05in

\noindent	
{\bf Step 1.}
We argue by induction on $d$. First, for $d=m$, Proposition \ref{cor-coeff} implies
$$
\sigma(\mu_{Q^n}\circ f^{-1}, t)\le C mt^{1/m}.
$$
Thus, the base case $d=m$ of the induction holds.

\vskip .05in

\noindent	
{\bf Step 2.}
We now make the inductive step.
Let $V$ be a random variable, independent
of the random vector~$U$, and uniformly distributed on $[-\frac{1}{2}, \frac{1}{2}]$.
We observe that, for any $\varepsilon>0$,
one has
\begin{equation}\label{eq-split}
\mathbb{E}\, \varphi'\bigl(f(U)\bigr)
= \mathbb{E}\bigl[\varphi'\bigl(f(U)\bigr)-\varphi'\bigl(f(U) + \varepsilon V\bigr)\bigr]
+
\mathbb{E}\, \varphi'\bigl(f(U) + \varepsilon V\bigr).
\end{equation}
For the first term, by \eqref{T-equiv} we obtain
\begin{align}\label{eq-first}
\mathbb{E}\bigl[\varphi'\bigl(f(&U)\bigr)-\varphi'\bigl(f(U) + \varepsilon V\bigr)\bigr]
\\
&\le
\|\varphi'\|_\infty\, \mathbb{E}_V\Bigl[\int_\mathbb{R}
|\varrho_{\mu_{Q^n}\circ f^{-1}}(t) - \varrho_{\mu_{Q^n}\circ f^{-1}}(t - \varepsilon V)|\, dt\Bigr]\nonumber
\\
&\le
2\|\varphi'\|_\infty\, \mathbb{E}_V\bigl[\sigma(\mu_{Q^n}\circ f^{-1}, \varepsilon |V|)\bigr]\nonumber
\\
&\le
2\|\varphi'\|_\infty\, \sigma\bigl(\mu_{Q^n}\circ f^{-1}, \varepsilon\cdot \mathbb{E}|V|\bigr)
\le 2\|\varphi'\|_\infty\, \sigma(\mu_{Q^n}\circ f^{-1}, \varepsilon),\nonumber
\end{align}
where, in the last two steps, we have used the concavity and monotonicity of the function
$\sigma(\mu_{Q^n}\circ f^{-1}, \cdot)$
(see Lemma 2.1 in \cite{Kos-MS}).

\vskip .05in

\noindent	
{\bf Step 3.}
To estimate the second term in~\eqref{eq-split}, namely
$\mathbb{E}\, \varphi'\bigl(f(U) + \varepsilon V\bigr)$, 
we may, without loss of generality, assume that $k_n \ge 1$, 
and regard the polynomial $f$ as a polynomial in $x_n$:
$$
f(x_1,\ldots, x_{n-1},x_n) = \sum_{j=0}^{m} f_j(x_1,\ldots, x_{n-1})\, x_n^j.
$$
Let $\widetilde{U}_n = (U_1, \ldots, U_{n-1})$ for brevity. From Theorem~\ref{CW-T} (see Remark~\ref{rem-CW-T}), applied to $f(x_1,\ldots,x_{n-1},U_n,V)$, we obtain
\begin{align*}
\mathbb{E}\, \varphi'\bigl(f(U) + \varepsilon V\bigr)
&=
\mathbb{E}_{\widetilde{U}_n} \mathbb{E}_{U_n, V}\,
\varphi'\bigl(f(\widetilde{U}_n, U_n) + \varepsilon V\bigr)
\\
&\le
Cm\|\varphi\|_\infty^{1/m}\|\varphi'\|_\infty^{1-1/m}
\mathbb{E}_{\widetilde{U}_n}
\bigl(\mathbb{D}_{U_n}f(\widetilde{U}_n, U_n)+\tfrac{\varepsilon^2}{12}\bigr)^{-1/(2m)}.
\end{align*}
Furthermore, applying consecutively 
the Markov-type inequality in $L^2$ (see~\cite[Theorem~1.7.7]{MMR}) 
and the Nikolskii-type inequality for algebraic polynomials 
(see~\cite[Theorem~2.6]{DL93}), we obtain
\begin{align*}
\mathbb{D}_{U_n}f(\widetilde{U}_n, U_n)&\ge (c_1m^4)^{-k_n}\mathbb{E}_{U_n}\Bigl|\frac{\partial^{k_n}}{\partial x_n^{k_n}}f(\widetilde{U}_n, U_n)\Bigr|^2
\\
&\ge (c_1m^4)^{-k_n}(c_2m)^{-2}\Bigl|\frac{\partial^{k_n}}{\partial x_n^{k_n}}f(\widetilde{U}_n, 0)\Bigr|^2
\\ 
&=
(c_1m^4)^{-k_n}(k_n!)^2(c_2m)^{-2}\bigl|f_{k_n}(\widetilde{U}_n)\bigr|^2
\\
&\ge 
(c_1m^2)^{-k_n}\Bigl(\frac{k_n}{em}\Bigr)^{2k_n}(c_2m)^{-2}\bigl|f_{k_n}(\widetilde{U}_n)\bigr|^2
\\
&\ge 
(c_3m^2)^{-m}\bigl|f_{k_n}(\widetilde{U}_n)\bigr|^2,
\end{align*}
where in the last step, we used the 
estimate
$\bigl(\frac{k}{em}\bigr)^k\ge e^{-m}$
for $k\le m$.
Without loss of generality, we may assume that $c_3\ge1$.
Thus, we have
\begin{align*}
\mathbb{E}\, \varphi'\bigl(f(U) &+ \varepsilon V\bigr)
\\
&\le
Cm\|\varphi\|_\infty^{1/m}\|\varphi'\|_\infty^{1-1/m}
\mathbb{E}_{\widetilde{U}_n}\Bigl[ \bigl((c_3m^2)^{-m}|f_{k_n}(\widetilde{U}_n)|^2 + \tfrac{\varepsilon^2}{12}\bigr)^{-1/(2m)}\Bigr]
\\
&\le
c_4m^2\|\varphi\|_\infty^{1/m}\|\varphi'\|_\infty^{1-1/m}
\mathbb{E}_{\widetilde{U}_n}\bigl[ \bigl(|f_{k_n}(\widetilde{U}_n)|^2 + \varepsilon^2\bigr)^{-1/(2m)}\bigr].
\end{align*}
Furthermore, by Theorem \ref{T-meas}, we have
\begin{align*}
\mathbb{E}_{\widetilde{U}_n}\bigl[ \bigl(|f_{k_n}(\widetilde{U}_n)|^2 +& \varepsilon^2\bigr)^{-1/(2m)}\bigr]
\\
&=
\int_{0}^{\varepsilon^{-1/m}}
P\Bigl(\bigl(|f_{k_n}(\widetilde{U}_n)|^2 + \varepsilon^2\bigr)^{-1/2m}\ge \tau\Bigr)\, d\tau
\\
&=
\frac{1}{m}\int_{0}^{\infty}\frac{s}{(s^2+\varepsilon^2)^{1+1/2m}}\,
P\bigl(|f_{k_n}(\widetilde{U}_n)|\le s\bigr)\, ds
\\
&\le
\frac{4}{m}\int_{0}^{\infty}\frac{1}{(s+\varepsilon)^{1+1/m}}\,
P\bigl(|f_{k_n}(\widetilde{U}_n)|\le s+\varepsilon\bigr)\, ds,
\\
&\le
\frac{8}{m}\int_{0}^{\infty}\frac{1}{(s+\varepsilon)^{1+1/m}}\,
\sigma(\mu_{Q^{n-1}}\circ f_{k_n}^{-1}, s+\varepsilon)\, ds,
\end{align*}
where, in the last step, we applied Theorem~\ref{T-meas} and used
the estimate 
$$
\sigma(\mu_{Q^{n-1}}\circ f_{k_n}^{-1}, 2t)\le 2\sigma(\mu_{Q^{n-1}}\circ f_{k_n}^{-1}, t)\quad \forall t>0,
$$ 
which follows from the concavity of the function 
$\sigma(\mu_{Q^{n-1}}\circ f_{k_n}^{-1}, \cdot)$.
Since 
$$
d[f_{k_n}]\le d-1 \hbox{ and } [f_{k_n}]_\infty=[f]_\infty=1,
$$
we can apply the inductive hypothesis
and obtain the estimate
$$
\sigma(\mu_{Q^{n-1}}\circ f_{k_n}^{-1}, s+\varepsilon)
\le
C(m, d-1)(s+\varepsilon)^{1/m}\bigl(|\ln (s+\varepsilon)|^{d-1-m}+1\bigr)
$$
Therefore, for $\varepsilon\in(0, 4^{-1}]$, we have
\begin{align*}
\int_{0}^{\infty}\frac{1}{(s+\varepsilon)^{1+1/m}}\,
\sigma(\mu_{Q^{n-1}}\circ f_{k_n}^{-1}, s+\varepsilon)\, ds\quad\quad\quad\quad\quad&	
\\
\le
C(m, d-1)\int_{0}^{1-\varepsilon}
\frac{(s+\varepsilon)^{1/m}\bigl(|\ln (s+\varepsilon)|^{d-1-m}+1\bigr)}{(s+\varepsilon)^{1+1/m}}&\, ds
\\+
\int_{1-\varepsilon}^{\infty}&\frac{1}{(s+\varepsilon)^{1+1/m}}\, ds
\\
=
C(m, d-1)\int_{0}^{1-\varepsilon}
\frac{|\ln (s+\varepsilon)|^{d-1-m}+1}{s+\varepsilon}\, ds
+
m\quad\quad&
\\
=C(m, d-1)
\bigl((d-m)^{-1}|\ln\varepsilon|^{d-m} +|\ln\varepsilon|\bigr)+m\quad\quad\ &
\\
\le
2C(m, d-1)|\ln \varepsilon|^{d-m}+m.\quad\quad\quad\quad\quad\quad\quad\quad\quad\quad\ &
\end{align*}
Thus, for $\varepsilon\in(0, 4^{-1}]$,
$$
\mathbb{E}\, \varphi'\bigl(f(U) + \varepsilon V\bigr)
\le
c_5m\|\varphi\|_\infty^{1/m}\|\varphi'\|_\infty^{1-1/m}\bigl(2C(m, d-1) |\ln\varepsilon|^{d-m} + m\bigr),
$$
and, by combining this estimate with \eqref{eq-first} and using \eqref{eq-split},
we obtain 
\begin{align}\label{eq-inter}
\sigma(\mu_{Q^n}\circ & f^{-1}, t)
\\
&\le
2 \sigma(\mu_{Q^n}\circ f^{-1}, \varepsilon)
+
c_5mt^{1/m}\bigl(2C(m, d-1) |\ln\varepsilon|^{d-m} + m\bigr)\nonumber
\end{align}
for any $t>0$ and $\varepsilon\in(0, 4^{-1}]$.

\vskip .05in

\noindent	
{\bf Step 4.}
Assume that $t\in(0,4^{-d}]$.
Then
$$
\sigma(\mu_{Q^n}\circ f^{-1}, t)
=
\sum\limits_{k=1}^{\infty}2^{k-1}\Bigl[\sigma\bigl(\mu_{Q^n}\circ f^{-1}, t^{k}\bigr) -
2\sigma\bigl(\mu_{Q^n}\circ f^{-1}, t^{k+1}\bigr)\Bigr],
$$
since, by Theorem~\ref{CW-T}, we have
\begin{align*}
2^k\sigma\bigl(\mu_{Q^n}\circ f^{-1}, t^{k+1}\bigr)
&\le
Cd[\mathbb{D}f(U)]^{-1/2d}\, t^{(k+1)/d}2^{k}
\\
&\le Cd[\mathbb{D}f(U)]^{-1/2d}\, 4^{-1}2^{-k}\to0
\end{align*}
as $k\to +\infty$.
Hence, by combining the above estimate with the inequality \eqref{eq-inter}, we obtain
\begin{align*}
\sigma(\mu_{Q^n}\circ f^{-1}, t)
\le
c_5m\sum\limits_{k=1}^{\infty}
2^{k-1}t^{k/m}\bigl(2C(m, d-1) (k+1)^{d-m}&|\ln t|^{d-m} + m\bigr)
\\
\le c_5m\Bigl(2C(m, d-1)
\sum\limits_{k=1}^{\infty}
2^{-(k-1)}(k+1)^{d-m}\Bigr)t^{1/m}|\ln t|^{d-m}& 
\\
+ c_5m^2&\sum\limits_{k=1}^{\infty}
2^{-(k-1)} t^{1/m}
\\
\le
\Bigl(2c_5mC(m, d-1)
\sum\limits_{k=0}^{\infty}
2^{-k}(k+2)^{d-m}\Bigr)t^{1/m}|\ln t|^{d-m} + 2&c_5m^2 t^{1/m}.
\end{align*}
Thus,
$$
\sigma(\mu_{Q^n}\circ f^{-1}, t)
\le
C(m,d)t^{1/m}\bigl(|\ln t|^{d-m}+1\bigr)
$$
for $t\in(0, 4^{-d}]$.

\vskip .05in

\noindent
For $t\ge 4^{-d}$, we have
$$
\sigma(\mu_{Q^n}\circ f^{-1}, t)\le 1\le  4^{d/m} t^{1/m}\bigl(|\ln t|^{d-m} + 1\bigr).
%\le 4^{d-m+1} t^{1/m}\bigl(|\ln t|^{d-m} + 1\bigr).
$$
This completes the proof.
\end{proof}

\begin{remark}
It follows from the proof that, in the theorem above, one can take
$$
C(m, d) = (C\cdot(d-m+1))^{(d-m+1)^2}m^{d-m+1}
$$
for some universal constant $C\ge 1$.
\end{remark}

\section{Polynomial images of product measures}\label{sec-prod}

The proofs of the main results in this section are based on a key observation due to Bobkov, Chistyakov, and G\"otze
(see \cite[Lemma 4.3]{BChG}),
concerning a representation of any probability density $\varrho$ on $\mathbb{R}$ of bounded variation 
as a convex mixture of uniform distributions.  
Namely, a probability density $\varrho$ on $\mathbb{R}$ is said to be represented as a convex mixture of uniform distributions
with a mixing measure $\pi$ if $\pi$ is a probability measure
on the half-plane $\{a < b\}$ such that 
$$
\int_\mathbb{R}\varphi(x)\varrho(x)\, dx = \int_{\{a<b\}} \int_{\mathbb{R}}\varphi(x)\varrho_{[a, b]}(x)\, dx\, \pi(dadb)
\quad \forall \varphi\in C_b^\infty(\mathbb{R}),
$$
where
$$
\varrho_{[a, b]}=
\tfrac{1}{b-a}I_{[a, b]}.
$$

\begin{theorem}[see {\cite[Lemma 4.3]{BChG}}]\label{lem-BChG}
Any probability density $\varrho$
on $\mathbb{R}$ of bounded variation
can be represented as a convex
mixture
$$
\varrho=\int_{\{a<b\}} \varrho_{[a,b]}\, \pi(dadb)
$$
of uniform distributions with a mixing measure $\pi$,
such that
$$
\|D_1\varrho\|_{\rm TV}=\int_{\{a<b\}} \|D_1\varrho_{[a, b]}\|_{\rm TV}\, \pi(dadb).
$$
\end{theorem}

\noindent
We recall that $\|D_1\varrho_{[a, b]}\|_{\rm TV} = \frac{2}{b-a}$.

\vskip .1in

The first main result of this section is as follows.

\begin{theorem}\label{Coeff-T-1} 
There exists an absolute constant $C>0$ such that for all $d, n\in\mathbb{N}$, for any functions 
$\varrho_1, \ldots, \varrho_n\colon \mathbb{R}\to\mathbb{R}$ of bounded variation
satisfying
$$
\|\varrho_j\|_{L^1(\mathbb{R})} = 1 \quad \forall j \in \{1, \dots, n\},
$$
for any non-constant polynomial $f\in\mathcal{P}_d(\mathbb{R}^n)$, 
and for any $\varphi\in C_b^\infty(\mathbb{R})$,
one has
\begin{align*}
\int_{\mathbb{R}^n} \varphi'(f(x))&\varrho_1(x_1)\ldots\varrho_n(x_n)\, dx
\\
&\le
C\min\{d, n\}\bigl(1+\max_{1\le j\le n}\|D_1\varrho_j\|_{\rm TV}\bigr)[f]_2^{-1/d}\|\varphi\|_\infty^{1/d}\|\varphi'\|_\infty^{1-1/d}.
\end{align*}
\end{theorem}

\begin{proof}
By scaling,
we may assume that $[f]_2=1$ and $\|\varphi'\|_\infty=1$.
Let 
$$
M:=\max\limits_{1\le j\le n}\|D_1\varrho_j\|_{\rm TV}.
$$
For each $j$, define
$$
\varrho_j^{1}:=\max\{\varrho_j, 0\},\quad 
\varrho_j^{-1}:=\max\{-\varrho_j, 0\},\quad
I_j^{\pm1}:=\int_{\mathbb{R}}\varrho_j^{\pm1}(t)\, dt.
$$
We note that the norm $\|D_1\varrho_j^{\pm 1}\|_{\rm TV}$ can be controlled by 
$\|D_1\varrho_j\|_{\rm TV}$.
Indeed, using \eqref{T-equiv}, and noting that the function 
$t\mapsto \max\{t, 0\}$ is $1$-Lipschitz,
we can write
\begin{align*}
&
\|D_1\varrho_j^{\pm1}\|_{\rm TV}
= \sup_{t>0}t^{-1}\sigma(\varrho_j^{\pm1}, t)
\le 6\sup_{t>0}t^{-1}\omega(\varrho_j^{\pm1}, t)
\\
&\le 6\sup_{t>0}t^{-1}\omega(\varrho_j, t)
\le 12\sup_{t>0}t^{-1}\sigma(\varrho_j, t)
= 12\|D_1\varrho_j\|_{\rm TV}\le 12M.
\end{align*}
Fixing an arbitrary choice of signs 
$\varepsilon_1, \ldots, \varepsilon_n \in \{-1, 1\}$, 
and applying Theorem~\ref{lem-BChG} to the 
densities $\frac{1}{I_j^{\varepsilon_j}}\varrho_j^{\varepsilon_j}$, when the integrals $I_j^{\varepsilon_j}$ are nonzero, we can write
\begin{align*}
&\int_{\mathbb{R}^n}\varphi'(f(x))\varrho_1^{\varepsilon_1}(x_1)\cdot\ldots\cdot\varrho_n^{\varepsilon_n}(x_n)\, dx
\\
&=\prod_{j=1}^n I_j^{\varepsilon_j}
\int_{\{a_1<b_1\}}\ldots \int_{\{a_n<b_n\}}
I({\bf a}, {\bf b})\, \pi_1^{\varepsilon_1}(da_1 db_1)\ldots\pi_n^{\varepsilon_n}(da_ndb_n),
\end{align*}
where ${\bf a}=(a_1, \ldots, a_n)$, 
${\bf b}=(b_1, \ldots, b_n)$, and
$$
I({\bf a}, {\bf b})=
\prod_{j=1}^n\tfrac{1}{b_j-a_j}\int_{\prod_{j=1}^n[a_j, b_j]}
\varphi'(f(x))\, dx.
$$
Let
\begin{equation}\label{eq-change}
Ly:=\Bigl(\tfrac{a_1+b_1}{2}+y_1(b_1-a_1), \ldots, \tfrac{a_n+b_n}{2}+y_n(b_n-a_n)\Bigr),
\quad
g(y): = f(Ly).
\end{equation}
After the change of variables $x=Ly$, we obtain
$$
\prod_{j=1}^n\tfrac{1}{b_j-a_j}\int_{\prod_{j=1}^n[a_j, b_j]}
\varphi'(f(x))\, dx
= 
\int_{Q^n}\varphi'(g(y))\, dy.
$$
We now estimate $[g]_2$.
If 
$$
f(x):=\sum_{k=0}^{d(f)} f_k(x),\quad
f_k(x)=
\sum_{j_1+\ldots+j_n = k}c_{j_1,\ldots, j_n}x_1^{j_1}\ldots x_n^{j_n},
$$
then
$$
g(y):=\sum_{k=0}^{d(f)}g_k(y),
\ \ 
g_k(y)=
\sum_{j_1+\ldots+j_n = k}\tilde{c}_{j_1,\ldots, j_n}x_1^{j_1}\ldots x_n^{j_n},\ k=0, 1, \ldots d(f)-1,
$$
for some $\tilde{c}_{j_1, \ldots, j_n}\in \mathbb{R}$,
and
\begin{equation}\label{eq-coeff}
g_{d(f)}(y) = 
\sum_{j_1+\ldots+j_n = d(f)}c_{j_1,\ldots, j_n}(b_1-a_1)^{j_1}\ldots(b_n-a_n)^{j_n}y_1^{j_1}\ldots y_n^{j_n}.
\end{equation}
Thus,
\begin{align*}
[g]_2^{-1/d(f)} &= \Bigl(\sum_{j_1+\ldots+j_n = d(f)}c_{j_1,\ldots, j_n}^2(b_1-a_1)^{2j_1}\ldots(b_n-a_n)^{2j_n} \Bigr)^{-1/2d(f)}
\\
&\le
\Bigl(\sum_{j_1+\ldots+j_n = d(f)}c_{j_1,\ldots, j_n}^2(b_1-a_1)^{-j_1/d(f)}\ldots(b_n-a_n)^{-j_n/d(f)} \Bigr),
\end{align*}
where we have used the convexity of the function $t\mapsto t^{-1/(2d(f))}$
on $(0, +\infty)$.
Using the convexity of the exponent, we deduce
$$
(b_1-a_1)^{-j_1/d(f)}\ldots(b_n-a_n)^{-j_n/d(f)}
\le
\frac{j_1}{d(f)}(b_1-a_1)^{-1}+\ldots+\frac{j_n}{d(f)}(b_n-a_n)^{-1}.
$$
Proposition~\ref{cor-coeff} now yields
\begin{align*}
\int_{Q^n}\varphi'(&g(y))\, dy
\\
&\le 
C\min\{d, n\}\|\varphi\|_\infty^{1/d(f)}
\Bigl(\sum_{j_1+\ldots+j_n = d(f)}c_{j_1,\ldots, j_n}^2\sum_{k=1}^n\frac{j_k}{d(f)}(b_k-a_k)^{-1} \Bigr).
\end{align*}
Thus, we have the estimate
\begin{align*}
&\int_{\mathbb{R}^n}\varphi'(f(x))\varrho_1^{\varepsilon_1}(x_1)\cdot\ldots\cdot\varrho_n^{\varepsilon_n}(x_n)\, dx
\\
&\le
6C\min\{d, n\}\|\varphi\|_\infty^{1/d(f)}
\prod_{j=1}^n I_j^{\varepsilon_j}
\Bigl(\sum_{j_1+\ldots+j_n = d(f)}c_{j_1,\ldots, j_n}^2\sum_{k=1}^n \frac{j_k}{d(f)}\frac{M}{I_k^{\varepsilon_k}} \Bigr)
\\
&\le
6CM\min\{d, n\}\|\varphi\|_\infty^{1/d(f)}
\Bigl(\sum_{j_1+\ldots+j_n = d(f)}c_{j_1,\ldots, j_n}^2\sum_{k=1}^n \frac{j_k}{d(f)}\prod_{j\ne k} I_j^{\varepsilon_j} \Bigr).
\end{align*}
This estimate remains valid even if one of the integrals $I_j^{\varepsilon_j}$ vanishes.
Finally, we observe that
\begin{align*}
&\int_{\mathbb{R}^n}\varphi'(f(x))\varrho_1(x_1)\cdot\ldots\cdot\varrho_n(x_n)\, dx
\\
&= \sum_{\varepsilon_1, \ldots, \varepsilon_n\in\{-1, 1\}}\prod_{j=1}^n\varepsilon_j
\int_{\mathbb{R}^n}\varphi'(f(x))\varrho_1^{\varepsilon_1}(x_1)\cdot\ldots\cdot\varrho_n^{\varepsilon_n}(x_n)\, dx
\\
&\le
6CM\min\{d, n\}\|\varphi\|_\infty^{1/d(f)}
\!\!\!
\sum_{\varepsilon_1, \ldots, \varepsilon_n\in\{-1, 1\}}
\sum_{j_1+\ldots+j_n = d(f)}c_{j_1,\ldots, j_n}^2\sum_{k=1}^n \frac{j_k}{d(f)}\prod_{j\ne k} I_j^{\varepsilon_j}
\\
&=
12CM\min\{d, n\}\|\varphi\|_\infty^{1/d(f)}
\sum_{j_1+\ldots+j_n = d(f)}c_{j_1,\ldots, j_n}^2\sum_{k=1}^n \frac{j_k}{d(f)}
\prod_{j\ne k}(I_j^{+1} + I_j^{-1})
\\
&\le 12CM\min\{d, n\}\|\varphi\|_\infty^{1/d(f)}.
\end{align*}
When $\|\varphi\|_\infty\le 1$,
we obtain
$$
\int_{\mathbb{R}^n}\varphi'(f(x))\varrho_1(x_1)\cdot\ldots\cdot\varrho_n(x_n)\, dx
\le 
12CM\min\{d, n\}\|\varphi\|_\infty^{1/d}.
$$
If $\|\varphi\|_\infty> 1$,
then
$$
\int_{\mathbb{R}^n}\varphi'(f(x))\varrho_1(x_1)\cdot\ldots\cdot\varrho_n(x_n)\, dx
\le  \|\varphi'\|_\infty=1\le \|\varphi\|_\infty^{1/d}.
$$
This completes the proof.
\end{proof}

\begin{remark}
In particular, Theorem~\ref{Coeff-T-1} implies Theorem~\ref{reg-prod-1}.
\end{remark}

\vskip .1in

We proceed to the second main result of this section.

\begin{theorem}\label{Coeff-T-2}
For $d, m\in \mathbb{N}$, $d\ge m$, there exists a constant $C(m, d)$,
depending only on these parameters, such that for any functions of bounded variation
$\varrho_1, \ldots, \varrho_n\colon \mathbb{R}\to\mathbb{R}$
satisfying 
$$
\|\varrho_j(t)\|_{L^1(\mathbb{R})}=1 \quad \forall j\in\{1, \ldots, n\},
$$ 
for any non-constant polynomial $f\in\mathcal{P}_{d, m}(\mathbb{R}^n)$, 
for any $t>0$ and any function $\varphi\in C_b^\infty(\mathbb{R})$ with $\|\varphi'\|_\infty\le 1$ and $\|\varphi\|_\infty\le t$,
one has
\begin{align*}
\int_{\mathbb{R}^n} &\varphi'(f(x))\varrho_1(x_1)\ldots\varrho_n(x_n)\, dx
\\
&\le 
C(m, d)\bigl(1+\max_{1\le j\le n}\|D_1\varrho_j\|_{\rm TV}\bigr)^{d/m} [f]_\infty^{-1/m} t^{1/m}\bigl(|\ln ([f]_\infty^{-1}t)|^{d-m}+1\bigr).
\end{align*}
\end{theorem}

\begin{proof}
Without loss of generality, we may assume that 
$[f]_\infty=1$.
If $t\ge 1$, then	
$$
\int_{\mathbb{R}^n}\varphi'(f(x))\varrho_1(x_1)\ldots \varrho_n(x_n)\, dx
\le 1\le  t^{1/m}\bigl(|\ln t|^{d-m}+1\bigr).
$$	
Hence, we consider $t \in (0, 1)$.
In this case,	
we argue similarly to the proof of Theorem~\ref{Coeff-T-1}.
Let $M:=\max\limits_{1\le j\le n}\|D_1\varrho_j\|_{\rm TV}$,
and let 
$$
\varrho_j^{1}:=\max\{\varrho_j, 0\},\quad 
\varrho_j^{-1}:=\max\{-\varrho_j, 0\},\quad
I_j^{\pm1}:=\int_{\mathbb{R}}\varrho_j^{\pm1}(t)\, dt.
$$
Now, for any fixed choice of signs
$\varepsilon_1, \ldots, \varepsilon_n\in\{-1, 1\}$, we apply
Theorem~\ref{lem-BChG} and obtain
\begin{align*}
&\int_{\mathbb{R}^n}\varphi'(f(x))\varrho_1^{\varepsilon_1}(x_1)\cdot\ldots\cdot\varrho_n^{\varepsilon_n}(x_n)\, dx
\\
&=
\prod_{j=1}^n I_j^{\varepsilon_j}
\int_{\{a_1<b_1\}}\ldots \int_{\{a_n<b_n\}}
\int_{Q^n}
\varphi'(g(y))\, dy\, \pi_1^{\varepsilon_1}(da_1 db_1)\ldots\pi_n^{\varepsilon_n}(da_ndb_n),
\end{align*}
where $g(y) = f(Ly)$, and $L$ denotes the change of variables defined in \eqref{eq-change}.

\vskip .1in

\noindent
Let $1=[f]_\infty=a_{j_1, \ldots, j_n}$,
$j_1+\ldots+j_n=d(f)$.
If $t\ge [g]_\infty$, then
$$
\int_{Q^n}\varphi'(g(y))\, dy
\le 1\le [g]_\infty^{-1/m}t^{1/m}
\bigl(|\ln t|^{d-m}+1\bigr).
$$
If $t < [g]_\infty$, then by Theorem~\ref{ind-deg}, we obtain the estimate
$$
\int_{Q^n}\varphi'(g(y))\, dy
\le 
C(m, d)[g]_\infty^{-1/m}t^{1/m}\bigl(|\ln([g]_\infty^{-1}t)|^{d-m}+1\bigr).
$$ 
We point out that in the case $t < [g]_\infty$
one has
$$
|\ln([g]_\infty^{-1}t)|
=-\ln t + \ln [g]_\infty.
$$
If $[g]_\infty<1$ then 
$$
|\ln([g]_\infty^{-1}t)|^{d-m}\le |\ln t|^{d-m}
$$
and
$$
\int_{Q^n}\varphi'(g(y))\, dy
\le 
C(m, d)[g]_\infty^{-1/m}t^{1/m}\bigl(|\ln t|^{d-m}+1\bigr).
$$
Finally, if $[g]_\infty \ge 1$ then
\begin{align*}
&[g]_\infty^{-1/m}t^{1/m}\bigl(|\ln([g]_\infty^{-1}t)|^{d-m}+1\bigr)
\\
&\le 
2^{d-m}[g]_\infty^{-1/m}t^{1/m}\bigl(|\ln t|^{d-m}+1\bigr)
+2^{d-m}
[g]_\infty^{-1/m}t^{1/m}(\ln[g]_\infty)^{d-m}
\\
&\le 
2^{d-m}[g]_\infty^{-1/m}t^{1/m}\bigl(|\ln t|^{d-m}+1\bigr)
+2^{d-m}m^{d-m}(d-m)^{d-m}t^{1/m}
\end{align*}
and in all cases we have
$$
\int_{Q^n}\varphi'(g(y))\, dy
\le 
C_1(m, d)([g]_\infty^{-1/m}+1)t^{1/m}\bigl(|\ln t|^{d-m}+1\bigr).
$$
By \eqref{eq-coeff},
$$
[g]_\infty\ge (b_1-a_1)^{j_1}\ldots(b_n-a_n)^{j_n},
$$
which implies
\begin{align*}
\int_{Q^n}&\varphi'(g(y))\, dy
\\
&\le 
C_1(m, d) \bigl(1+(b_1-a_1)^{-j_1/m}\ldots(b_n-a_n)^{-j_n/m}\bigr)
t^{1/m}
\bigl(|\ln t|^{d-m}+1\bigr)
\end{align*}
and
\begin{align*}
&\int_{\mathbb{R}^n}\varphi'(f(x))\varrho_1^{\varepsilon_1}(x_1)\cdot\ldots\cdot\varrho_n^{\varepsilon_n}(x_n)\, dx
\\	
&\le
C_1(m, d)t^{1/m}
\bigl(|\ln t|^{d-m}+1\bigr)
\prod_{j=1}^n I_j^{\varepsilon_j}
\\
&
\times
\int_{\{a_1<b_1\}}\ldots \int_{\{a_n<b_n\}}
\Bigl(1+\prod_{k=1}^n(b_k-a_k)^{-j_k/m}\Bigr)
\, \pi_1^{\varepsilon_1}(da_1 db_1)\ldots\pi_n^{\varepsilon_n}(da_ndb_n)
\\
&=
C_1(m, d)t^{1/m}
\bigl(|\ln t|^{d-m}\!+\!1\bigr)\!
\prod_{j=1}^n I_j^{\varepsilon_j}\!
\Bigl(1\!+\!\prod_{k=1}^n\int(b_k-a_k)^{-j_k/m}\, \pi_k^{\varepsilon_k}(da_k db_k)\Bigr).
\end{align*}
Applying H\"older's inequality, we obtain
\begin{align*}
&\int_{\{a_k<b_k\}}(b_k-a_k)^{-j_k/m}\, \pi_k^{\varepsilon_k}(da_k db_k)
\\
&\le 
\Bigl(\int_{\{a_k<b_k\}}(b_k-a_k)^{-1}\, \pi_k^{\varepsilon_k}(da_k db_k)
\Bigr)^{j_k/m}
\\
&\le (\|D_1\varrho_k^{\varepsilon_k}\|_{\rm TV}/I_k^{\varepsilon_k})^{j_k/m}
\le (12M)^{j_k/m}(I_k^{\varepsilon_k})^{-j_k/m}.
\end{align*}
Thus,
\begin{align*}
\int_{\mathbb{R}^n}&\varphi'(f(x))\varrho_1^{\varepsilon_1}(x_1)\cdot\ldots\cdot\varrho_n^{\varepsilon_n}(x_n)\, dx
\\	
&\le
C_2(m, d)t^{1/m}
\bigl(|\ln t|^{d-m}+1\bigr)
\Bigl(\prod_{j=1}^n I_j^{\varepsilon_j}+\prod_{k=1}^n (I_k^{\varepsilon_k})^{1-j_k/m}
M^{d(f)/m}\Bigr).
\end{align*}
Summing over all choices of signs, we arrive at
\begin{align*}
&\int_{\mathbb{R}^n}\varphi'(f(x))\varrho_1(x_1)\cdot\ldots\cdot\varrho_n(x_n)\, dx
\\
&= \sum_{\varepsilon_1, \ldots, \varepsilon_n\in\{-1, 1\}}\prod_{j=1}^n\varepsilon_j
\int_{\mathbb{R}^n}\varphi'(f(x))\varrho_1^{\varepsilon_1}(x_1)\cdot\ldots\cdot\varrho_n^{\varepsilon_n}(x_n)\, dx
\\
&\le
C_2(m, d)t^{1/m}
\bigl(|\ln t|^{d-m}+1\bigr)\!\!\!\!\!\!\!
\sum_{\varepsilon_1, \ldots, \varepsilon_n\in\{-1, 1\}}\!\!
\Bigl(\prod_{j=1}^n I_j^{\varepsilon_j}\!+\!
\prod_{k=1}^n (I_k^{\varepsilon_k})^{1-j_k/m}
M^{d(f)/m}\Bigr).
\end{align*}
We note that
$$
\sum_{\varepsilon_1, \ldots, \varepsilon_n\in\{-1, 1\}}
\prod_{j=1}^n I_j^{\varepsilon_j}
=\prod_{j=1}^n (I_j^{+1}+I_j^{-1})=1
$$
and
\begin{align*}
\sum_{\varepsilon_1, \ldots, \varepsilon_n\in\{-1, 1\}}
\prod_{k=1}^n (I_k^{\varepsilon_k})^{1-j_k/m}
&=
\prod_{k=1}^n \bigl((I_k^{+1})^{1-j_k/m}
+(I_k^{-1})^{1-j_k/m}\bigr)
\\
&\le 
\prod_{k=1}^n 
\Bigl(2 \Bigl(\frac{1}{2}I_k^{+1}
+\frac{1}{2}I_k^{-1}\Bigr)^{1-j_k/m}\Bigr)
=2^{d(f)/m},
\end{align*}
where we have used the concavity of the function $s\mapsto s^{1-j_k/m}$, $s>0$.
Thus,
\begin{align*}
\int_{\mathbb{R}^n}\varphi'(f(x))&\varrho_1(x_1)\cdot\ldots\cdot\varrho_n(x_n)\, dx
\\
&\le
C_2(m, d)t^{1/m}
\bigl(|\ln t|^{d-m}+1\bigr)
\bigl(1+
(2M)^{d(f)/m}\bigr),
\end{align*}
which completes the proof.
\end{proof}

\begin{remark}
Theorem~\ref{reg-prod-2} follows
directly from Theorem~\ref{Coeff-T-2}.
\end{remark}

\section*{Acknowledgements}

The author would like to thank Sergey Tikhonov for carefully reading the manuscript and for valuable comments and suggestions.

The research was supported by the AEI grants 
RYC2023-043616-I and
PID2023-150984NB-I00 funded by 
MICIU/AEI/10.13039/501100011033/FEDER, EU,
and by the Spanish State Research Agency, through the Severo Ochoa and Mar\'ia de Maeztu Program for Centers and
Units of Excellence in R\&D (CEX2020-001084-M).
The author thanks CERCA Programme (Generalitat de Catalunya) for institutional support.
{\sloppy

}


\begin{thebibliography}{}

\bibitem{AKCh}
Arkhipov, G.I., Karatsuba, A.A., Chubarikov,  V.N. Trigonometric integrals, 
\emph{Math. USSR-Izv.}, 15(2), 211--239 (1980).

\bibitem{AFP}
Ambrosio, L., Fusco, N., Pallara, D. 
Functions of bounded variation and free discontinuity problems. Courier Corporation (2000).
	
\bibitem{A-AGM-1}
Artstein-Avidan, S., Giannopoulos, A., Milman, V. D. Asymptotic geometric analysis, Part I (Vol. 202). American Mathematical Society (2015).
	
\bibitem{A-AGM-2}
Artstein-Avidan, S., Giannopoulos, A., Milman, V. D. Asymptotic geometric analysis, Part II (Vol. 261). American Mathematical Society  (2021).

\bibitem{BC14}
Bally, V., Caramellino, L. On the distances between probability density functions,
\emph{Electron. J. Probab.}, 19,
1--33,
(2014).

\bibitem{BC17}
Bally, V., Caramellino, L. Convergence and regularity of probability laws by using an interpolation method, \emph{Ann. Probab.}, 
45(2), 1110--1159.

\bibitem{BC19}
Bally, V., Caramellino, L.
Total variation distance between stochastic polynomials and invariance principles,
\emph{Ann. Probab.}, 47(6), 3762--3811 (2019).

\bibitem{BCP}
Bally, V., Caramellino, L., Poly, G.
Regularization lemmas and convergence in total variation,
\emph{Electron. J. Probab.}, 25, 1--20, (2020).
	
\bibitem{BIN}
Besov, O.V., Il'in, V.P., Nikolski{\u{\i}}, S.M.
Integral representations of functions and imbedding theorems.
V.~I,~II.
Winston \& Sons, Washington; Halsted Press, New York -- Toronto
-- London (1978, 1979).
	
\bibitem{BobkIsop} Bobkov, S.G.
Isoperimetric and analytic inequalities for log-concave probability measures.
\emph{Ann. Probab.}, 27(4), 1903--1921 (1999).
	
\bibitem{Bobk07}	
Bobkov, S.G. Large deviations and isoperimetry over convex probability measures with heavy tails,
\emph{Electron. J. Probab.}, 12, 1072--1100 (2007).	
	
\bibitem{BChG}
Bobkov, S.G., Chistyakov, G.P., G\"otze, F.
Fisher information and the central limit theorem.
\emph{Probab. Theory Related Fields}, 159(1-2), 1--59 (2014).
	
\bibitem{BKZ}
Bogachev, V.I., Kosov, E.D., Zelenov, G.I.
Fractional smoothness of distributions of polynomials and
a fractional analog of the Hardy--Landau--Littlewood inequality.
\emph{Trans. Amer. Math. Soc.},
370(6), 4401--4432 (2018).
	
%\bibitem{BKP}
%Bogachev, V.I., Kosov, E.D., Popova, S.N. A new approach to Nikolskii--Besov classes. \emph{
%Mosc. Math. J.}, 19(4), 619--654 (2019).
	
\bibitem{Bor74}
Borell, C. Convex measures on locally convex spaces.
\emph{Ark. Mat.}, 12, 239--252 (1974).
	
\bibitem{Bor75}
Borell, C. Convex set functions in d-space. \emph{Period. Math. Hungar.}, 6(2), 111--136  (1975).
	
\bibitem{BGVV}
Brazitikos, S., Giannopoulos, A., Valettas, P., Vritsiou, B.H.
Geometry of isotropic convex bodies. American Mathematical Society (2014).
	
\bibitem{CCW99}
Carbery, A., Christ, M., Wright, J. Multidimensional van der Corput and sublevel set estimates. \emph{J. Amer. Math. Soc.}, 12(4), 981--1015 (1999).
	
\bibitem{CarWr}
Carbery, A., Wright, J. Distributional and $L^q$ norm inequalities for polynomials
over convex bodies in $R^n$.
\emph{Math. Res. Lett.} 8(3), 233--248 (2001).
	
\bibitem{CW02}	
Carbery, A., Wright, J. What is van der Corput's lemma in higher dimensions? \emph{Publ. Mat.}, 13--26 (2002).	
	
\bibitem{CLTT05}	
Christ, M., Li, X., Tao, T., Thiele, C. On multilinear oscillatory integrals, nonsingular and singular. \emph{Duke Math. J.}, 130(2), 321--351 (2005).  
	
\bibitem{DL93}
DeVore, R.A., Lorentz, G.G. Constructive approximation (Vol. 303). Springer Science and Business Media (1993).
	
\bibitem{FrGue}
Fradelizi, M., Gu\'edon, O.
The extreme points of subsets of $s$-concave probabilities
and a geometric localization theorem.
\emph{Discrete Comput. Geom.}, 31(2), 327--335 (2004).

\bibitem{GGX18}
Gilula, M., Gressman, P.T., Xiao, L. 
Higher decay inequalities for multilinear oscillatory
integrals, 
\emph{Math. Res. Lett.}, 25(3) 819--842 (2018).

%\bibitem{GOX21}
%Gilula, M., O'Neill, K., Xiao, L. 
%Oscillatory Loomis--Whitney and projections of sublevel sets, 
%\emph{J. Anal. Math.}, 145(1), 307--333 (2021).

\bibitem{GM22} Glazer, I., Mikulincer, D.
Anti-concentration of polynomials: Dimension-free covariance bounds and decay of Fourier coefficients.
\emph{J. Funct. Anal.}, 283(9), 109639 (2022).

\bibitem{GPU}
G\"otze, F., Prokhorov, Yu.V., Ul'yanov, V.V.
Bounds for characteristic functions of polynomials in asymptotically normal random variables.
\emph{Russian Math. Surveys}, 51(2), 181--204 (1996).
	
\bibitem{Gr14}
Grafakos, L. Classical fourier analysis (Vol. 2). New York: Springer (2014).

\bibitem{GX16}
Gressman, P.T., Xiao, L. Maximal decay inequalities for trilinear oscillatory integrals of convolution type. \emph{J. Funct. Anal.}, 271(12), 3695--3726 (2016).
	
\bibitem{GM87}
Gromov, M., Milman, V.D. Generalization of the spherical isoperimetric inequality to uniformly convex Banach spaces. \emph{Compos. Math.},
62(3), 263--282 (1987).

\bibitem{KLS}
Kannan, R., Lov\'asz, L., Simonovits, M.
Isoperimetric problems for
convex bodies and a localization lemma.
\emph{Discrete Comput. Geom.}, 13, 541--559 (1995).

\bibitem{KT20}
Kolomoitsev, Yu.,  Tikhonov, S. Properties of moduli of smoothness in $L_p(\mathbb{R}^d)$. \emph{J. Approx. Theory}, 257, 105423.

\bibitem{Kos}
Kosov, E.D. 
Fractional smoothness of images of logarithmically concave measures under polynomials.
\emph{J. Math. Anal. Appl.}, 462(1), 390--406 (2018).

\bibitem{Kos-FCAA}
Kosov, E.D.
On fractional regularity of distributions of functions in Gaussian random variables.
\emph{Fract. Calc. Appl. Anal.}, 22(5), 1249--1268 (2019).

\bibitem{Kos-MS}
Kosov, E.D. 
Besov classes on finite and infinite dimensional spaces.
\emph{Mat. Sb.}, 210(5), 663--692 (2019).

\bibitem{Kos-IMRN}
Kosov, E.D. 
Total variation distance estimates via $L^2$-norm for polynomials in log-concave random vectors.
\emph{Int. Math. Res. Not.}, 2021(21),  16492--16508 (2021).

\bibitem{Kos-Adv}
Kosov, E.D.
Regularity of linear and polynomial images of Skorohod differentiable measures.
\emph{Adv. Math.}, 397(5), 108193 (2022).

\bibitem{Kos-FA}
Kosov, E.D.
Distributions of polynomials
in Gaussian random variables under
constraints on the powers of variables.
\emph{Funct. Anal. Appl.},
56(2), 101--109 (2022).

\bibitem{KosZh}
Kosov, E.D., Zhukova, A.K. 
Improved bounds for the total variation distance between stochastic polynomials, \emph{Stochastic Process. Appl.}, 170, 104279 (2024).
	
\bibitem{Klartag07}
Klartag, B.
Power-law estimates for the central limit theorem for convex sets. \emph{J. Funct. Anal.}, 245(1), 284--310 (2007).
	
\bibitem{KL24}
Klartag, B., Lehec, J. Affirmative Resolution of Bourgain's Slicing Problem using Guan's Bound.
\emph{Geom. Funct. Anal.},
35, 1147--1168 (2025). 
	
\bibitem{Krug} Krugova, E.P.
On translates of convex measures. \emph{Mat. Sb.}, 188(2), 227--236 (1997).
	
\bibitem{LS}
Lov\'asz, L., Simonovits, M. Random walks in a convex body and an improved
volume algorithm. \emph{Random Structures Algorithms}, 4(4), 359--412 (1993).

\bibitem{MOO}
Mossel, E., O'Donnell, R., Oleszkiewicz, K.
Noise stability of functions with low influences: variance and optimality,
\emph{Ann. of Math.}, 171(1), 295--341 (2010).

\bibitem{MMR}
Milovanovi\'c, G.V., Mitrinovi\'c, D.S.,  Rassias, Th.M. Topics in Polynomials: Extremal Problems, Inequalities, Zeros. World Scientific (1994).
	
\bibitem{NSV}
Nazarov, F., Sodin, M., Volberg, A. The geometric Kannan--Lovasz--Simonovits
lemma, dimension-free estimates for the distribution of the values of polynomials, and
the distribution of the zeros of random analytic functions.
\emph{St. Petersburg Math. J.}, 14(2), 351--366 (2003).
	
\bibitem{NNP}
Nourdin, I., Nualart, D., Poly, G.
Absolute continuity and convergence of densities for random vectors on Wiener chaos,
\emph{Electron. J. Probab.}, 18, 1--19 (2013).

\bibitem{NP13}
Nourdin, I, Poly, G.
Convergence in total variation on Wiener chaos,
\emph{Stochastic Process. Appl.}, 123(2), 651--674 (2013).

\bibitem{NP15}
Nourdin, I, Poly, G.
An invariance principle under the total variation distance, 
\emph{Stochastic Process. Appl.}, 125(6), 2190--2205 (2015).

	
%\bibitem{PS94}
%Phong, D. H., Stein, E. M. (1994). Operator versions of the van der Corput lemma and Fourier integral operators. Mathematical Research Letters, 1(1), 27--33.
		
\bibitem{Par07}	
Parissis, I. Oscillatory Integrals with Polynomial Phase, PhD thesis (2007).
	
\bibitem{PSS01}
Phong, D.H., Stein, E.M., Sturm, J. Multilinear level set operators, oscillatory integral operators, and Newton polyhedra. \emph{Math. Ann.}, 319(3), 573--596 (2001).
	
\bibitem{Stein}
Stein, E.
Singular integrals and differentiability properties of functions.
Princeton University Press, Princeton (1970).
	
%\bibitem{vdC}
%J.G. van der Corput,
%Zahlentheoretische Abschätzungen, Math. Ann., 84 (1921) 53--79.
	
\end{thebibliography}
\end{document}